\documentclass[11pt]{amsart}
\usepackage{lmodern}
\usepackage{amsmath, amsthm, amssymb, amsfonts}
\usepackage[normalem]{ulem}
\usepackage{hyperref}
\usepackage{enumitem} 

\usepackage{verbatim} 
\usepackage{longtable}

\usepackage{mathtools}
\DeclarePairedDelimiter\ceil{\lceil}{\rceil}
\DeclarePairedDelimiter\floor{\lfloor}{\rfloor}

\usepackage{tikz}
\usetikzlibrary{decorations.pathmorphing}
\tikzset{snake it/.style={decorate, decoration=snake}}

\usepackage[margin=2.5cm]{geometry}
\usepackage{caption}

\usepackage{tikz-cd}
\usetikzlibrary{arrows}
\usetikzlibrary{shapes.geometric}

\theoremstyle{plain}
\newtheorem{thm}{Theorem}[section]
\newtheorem{cor}[thm]{Corollary}
\newtheorem{lem}[thm]{Lemma}
\newtheorem{prop}[thm]{Proposition}
\newtheorem{conj}[thm]{Conjecture}

\theoremstyle{definition}
\newtheorem{defn}[thm]{Definition}
\newtheorem{example}[thm]{Example}

\theoremstyle{remark}
\newtheorem{rmk}[thm]{Remark}

\newcommand{\BA}{{\mathbb{A}}}

\newcommand{\BC}{{\mathbb{C}}}
\newcommand{\BD}{{\mathbb{D}}}

\newcommand{\BH}{{\mathbb{H}}}

\newcommand{\BP}{{\mathbb{P}}}
\newcommand{\BQ}{{\mathbb{Q}}}
\newcommand{\BR}{{\mathbb{R}}}

\newcommand{\BZ}{{\mathbb{Z}}}

\newcommand{\CD}{{\mathcal D}}

\newcommand{\CH}{{\mathcal H}}

\newcommand{\CM}{{\mathcal M}}

\newcommand{\CP}{{\mathcal P}}

\newcommand{\Fp}{{\mathfrak{p}}}

\DeclareFontFamily{OT1}{rsfs}{}
\DeclareFontShape{OT1}{rsfs}{n}{it}{<-> rsfs10}{}
\DeclareMathAlphabet{\curly}{OT1}{rsfs}{n}{it}

\newcommand\Ext{\operatorname{Ext}}
\newcommand\Hom{\operatorname{Hom}}

\newcommand\Spec{\operatorname{Spec}}

\usepackage{tikz}
\usepackage{lmodern}
\usetikzlibrary{decorations.pathmorphing}

\title{A non-holomorphic P=W phenomenon}
\author{Zili Zhang}
\address{School of Mathematical Sciences, Key Laboratory Intelligent Computing and Applications (Tongji University), Ministry of Education, Shanghai, China}
\email{zhangzili@tongji.edu.cn}
\date{\today}

\begin{document}

\begin{abstract}
   We prove the P=W identity for isolated cluster varieties of dimension 3 without the full rank hypothesis. These cluster varieties are generally singular, and the associated Lagrangian fibrations are not complex algebraic. On the P side, we construct the perverse truncation by a detailed analysis of the explicit real-analytic geometry of the Lagrangian fibration. On the W side, we construct a natural non-proper algebraic morphism from the cluster varieties and investigate the decomposition of the derived push-forward of the constant sheaf along this morphism within the derived category of mixed Hodge modules. In the P=W phenomenon for non-full rank isolated cluster varieties of dimension 3, both the curious hard Lefschetz property on the W side and the relative hard Lefschetz property on the P side fail.

\end{abstract}

\maketitle
\setcounter{tocdepth}{1}
\tableofcontents

\section{Introduction}

\subsection{Perverse filtrations and the P=W phenomenon}
Let $Y$ be a real analytic variety. For a fixed perversity function $\Fp$, the bounded derived category $D_c^b(Y)$ of constructible sheaves of $\BQ$-vector spaces on $Y$ is naturally equipped with a perverse $t$-structure $({^\Fp\CD}^{\le 0},{^\Fp\CD}^{\ge 0})$. Given a morphism $f:X\to Y$ of real analytic varieties, the associated perverse truncations 
\[
^\Fp\tau_{\le k}: D_c^b(Y)\to {^p\CD}^{\le k} ,~~k\in\BZ
\]
induces, for any complex $F\in D_c^b(X)$, an increasing filtration
\begin{equation}\label{00}
P_0H^*(X,F)\subset P_1H^*(X,F)\subset\cdots H^*(X,F).
\end{equation}
This filtration, denoted by \eqref{00}, is called the \emph{perverse filtration of $F$ associated with the morphism $f$ with respect to the perversity function $\Fp$}; see Section 2.3 for detailed discussions. In ths special case where $F=Rf_*IC_X$ or $Rf_*\BQ_X$, the corresponding filtration \eqref{00} are simply referred to as the perverse filtration associated with $f$. When $f:X\to Y$ is a smooth morphism of complex algebraic varieties, the perverse filtration with respect to the middle perversity function is well-understood and satisfies a Lefschetz type symmetry: any relatively ample class $\alpha\in H^2(X,\BQ)$ induces isomorphisms
\[
\textup{Gr}^P_{d-k} H^*(X,\BQ)\xrightarrow[\cong]{\cup\alpha^k}\textup{Gr}^P_{d+k} H^*(X,\BQ),~~k\ge0.
\]
We refer to \cite{BBD,dCM,dCM09,KS} for more details about derived categories of constructible sheaves and $t$-structures. See also Section 2.3 for more discussion on perverse filtrations.

The classical $P=W$ phenomenon was first observed in the context of Hitchin systems and character varieties \cite{dCHM}. Let $C$ be a smooth projective curve.  The moduli space $\CM_B$ of $\textup{GL}(n,\BC)$-representations of the fundamental group $\pi_1(C)$, called the character variety, is an affine scheme. On the other hand, the Dolbeault moduli space $\CM_D$ of degree 0 rank $n$ semi-stable Higgs bundles on $C$ is naturally equipped with a proper Hitchin map $h:\CM_D\to\BA$ onto an affine space. Simpson proves in \cite{Sim} that there exists a canonical homeomorphism $\CM_B\cong\CM_D$ between these quasi-projective varieties, called the non-abelian Hodge correspondence. Therefore, there is a canonical identification of the cohomology groups $H^*(\CM_D,\BQ)=H^*(\CM_B,\BQ)$. Under this identification, one may compare the mixed Hodge-theoretic weight filtration on $H^*(\CM_B,\BQ)$ with the perverse filtration on $H^*(\CM_D,\BQ)$ associated with the Hitchin map. The following remarkable $P=W$ identity was conjectured by de Cataldo, Hausel and Migliorini in \cite{dCHM}, proved recently by Maulik-Shen  \cite{MS}, Hausel-Mellit-Minets-Schiffmann  \cite{HMMS}, and Maulik-Shen-Yin \cite{MSY} via different approaches.
\begin{thm}[$P=W$]
Under the canonical identification $H^*(\CM_B,\BQ)=H^*(\CM_D,\BQ)$ via the non-abelian Hodge theory, the $P=W$ identity holds:
\[
P_kH^*(\CM_D,\BQ)=W_{2k}H^*(\CM_B,\BQ)=W_{2k+1}H^*(\CM_B,\BQ),~~k\ge0.
\]
\end{thm}

When the classical $P=W$ identity holds, the relative hard Lefschetz symmetry \eqref{00} induces, via the non-abelian Hodge correspondence, a symmetry on the character varieties known as the curious hard Lefschetz property. A variety $X$ satisfies the \emph{curious hard Lefschetz property} if there exists a $2$-form of mixed Hodge type $(2,2)$ such that
\begin{equation} \label{chl}
\textup{Gr}^W_{2\dim X-2k} H^*(X,\BQ)\xrightarrow[\cong]{\cup\beta^k}\textup{Gr}^W_{2\dim X+2k} H^*(X,\BQ),~~k\ge0.
\end{equation} 
In fact, the formal resemblance between the relative hard Lefschetz and the curious hard Lefschetz properties was one of the earliest hints leading to the discovery of the $P=W$ phenomenon.

\subsection{P=W for cluster varieties}
The $P=W$ phenomenon has also been observed in the setting of cluster varieties, which are complex affine varieties defined by certain combinatorial data. See \cite{FZ,LS} for cluster varieties and \cite{Z,Z2} for the corresponding $P=W$ phenomenon. In \cite{Z2}, we propose the following conjecture.
    
\begin{conj}\label{conj}
    Let $X$ be a cluster variety over $\BC$. Then there exists a real Lagrangian fibration $h:X\to \BR^{\dim_\BC X}$ such that the $P=W$ identity holds on the cohomology groups, i.e.
    \[
    P_kH^*(X(M),\BQ)=W_{2k} H^*(X(M),\BQ)=W_{2k+1}H^*(X(M),\BQ),~~ k\ge 0. 
    \]
    We further conjecture that an analogous $PI=WI$ identity holds, i.e. the same identity remains true when cohomology is replaced by intersection cohomology $IH^*(X,\BQ)$.
\end{conj}

The $P=W$ conjecture for cluster varieties is verified in \cite{Z} for all 2-dimensional cluster varieties, and in \cite{Z2} for isolated cluster varieties of full rank in arbitrary dimension. In these cases, the cluster varieties are smooth so the $P=W$ and $PI=WI$ coincide. For 2-dimensional case, the Lagrangian fibrations are constructed via the plumbing calculus of (real) 4-dimensional Lefschetz fibrations. In isolated cluster variety case, the Lagrangian fibration is defined combinatorially as follows. Let $M=(a_{ij})$ be an $m\times n$ integer matrix. The isolated cluster variety $X(M)$ is defined as the subvariety in 
\[
\BC^{2n+m}=\Spec\,\BC[x_1,\cdots,x_n,x_1',\cdots,x_n',z_1,\cdots,z_m]
\]
cut out by the equations
\begin{equation} \label{defn1}
x_jx_j'=\prod_{i=1}^m z_i^{a_{ij}}+1\textup{ for }1\le j\le n \textup{ and } z_1,\cdots,z_m\ne 0.
\end{equation}
After a suitable change of variables, the Lagrangian fibration $h:X(M)\to \BR^{n+m}$ is given by
\begin{equation} \label{defn2}
\begin{split}
  h(x_1,\cdots,x_n,x'_1&,\cdots,x'_n,z_1,\cdots,z_m)\\
  =&(|x_1|^2-|x_1'|^2,\cdots,|x_n|^2-|x_n'|^2,\log|z_1|,\cdots,\log|z_m|).
\end{split}
\end{equation}
When the matrix $M$ is of full column rank, we say that the cluster variety $X(M)$ is of full rank. It follows immediately from definition that $h$ is not a morphism of complex algebraic varieties. So standard results from complex algebraic geometry, such as the Beilinson-Berstein-Deligne-Gabber decomposition theorem, do not directly apply. Nevertheless, in the full rank case with $\dim_\BC X(M)$  even, the fibration $h$ admits a deformation of complex structures, which is an equivariant free finite quotient of a product of certain proper holomorphic elliptic fibrations. Since the perverse filtration is independent of the complex structure, it can be computed using algebraic methods. When $\dim_\BC X(M)$ is odd and $M$ is of full rank, the product $h\times p:X(M)\times\BC^*\to \BR^{m+n+1}$ underlies such a complex structures and the perverse filtration can also be computed. These cases are studied in \cite{Z2}.

The non-full rank case is considerably more sutble. The associated cluster varieties are not smooth in general. Although the map \eqref{defn2} remains a Lagrangian fibration, it is no longer a finite free quotient of products of 1 or 2 dimensional cluster varieties. Furthermore, when $\dim_\BC X$ is odd, the Lagrangian fibration are essentially non-holmorphic, \emph{i.e.} it cannot be represented as a product of a complex holomorphic map with the natural projection $\BC^*\to \BR$, even after an \'etale base change.

\subsection{Main results}
In this paper, we focus on the 3-dimensional isolated cluster varieties of non-full rank. This is the lowest-dimensional setting in which the cluster varieties become singular and the Lagrangian fibrations fail to underlie complex structures. Our study of the 3-dimensional case is intended to develop a systematic approach which can be extended to the higher dimensions. The main result is the following.
\begin{thm}[Theorem \ref{P=W}, \ref{PI=WI}]
   The $P=W$ and $PI=WI$ conjectures (Conjecture \ref{conj}) hold for 3-dimensional isolated cluster varieties, where the perverse filtration is defined associated with the map $h$ given by \eqref{defn2} with respect to the upper middle perversity function.
\end{thm}

Since the full rank case has already been treated in \cite{Z2}, it remains to prove the non-full rank case. Following the analysis in Section 2.4, such cluster varieties are of the form
\begin{equation} \label{dfn1}
X_{a,b}=\{ (x,x^{\prime},y,y^{\prime},z)\in \mathbb{C}^5: xx^{\prime}=z^a+1,\,yy^{\prime}=z^b+1,\,z\neq0\}
\end{equation}
for some integers $a,b$ not both zero. The corresponding Lagrangian fibration is defined as
\begin{equation} \label{dfn2}
\begin{aligned}
h: X_{a,b} &\rightarrow \mathbb{R}^{3} \\
(x,x^{\prime},y,y^{\prime},z) &\mapsto (|x|^{2}-|x^{\prime}|^{2},
|y|^{2}-|y^{\prime}|^{2},\log|z|).
\end{aligned}
\end{equation}
We omit $a,b$ when no confusion arises. On $P$ side, we exploit a stratification of the fibration $h$ and construct the perverse truncation using explicit geometric data such as the monodromy matrices around the singular values and the degenerations of singular fibers. On $W$ side, the natural cluster structure realizes the cluster variety as a $\BC^{*2}$-fibration over $\BC^*$, and we compute the weight filtration using mixed Hodge modules together with resolution of singularities. Finally, we identify the corresponding subspaces in both filtrations by certain topological invariants that are independent of the complex structures and of the definition of $h$. More detailed explanations of the main strategy and the ingredients on both the $W$ and $P$ side are provided below.

\subsubsection{W side} 
Every isolated cluster variety $X$ is equipped with a natural morphism $f:X\to(\BC^*)^m$. The mixed Hodge structure on $X$ can be computed via the derived push-forward along this non-proper morphism $f$.  In the case of $X_{a,b}$, we have a fibered product 
\[
\begin{tikzcd}
    X_{a,b}\arrow[r]\arrow[rd,"f_{a,b}"]\arrow[d]& X_a\arrow[d,"f_a"]\\
    X_b\arrow[r,"f_b"]& \BC^*,
\end{tikzcd}
\]
where $X_a\subset \BC^3$ is the cluster variety defined by $\{xx'=z^a+1,z\neq0\}$, and $\BC^*$ is parametrized by the coordinator $z\neq0$. This description enables us to study the decomposition property of the derived push-forward along the non-proper morphism $f_{a,b}$ via the ones of $f_a$ and $f_b$ and apply the relative K\"unneth formula in the category of the mixed Hodge modules. Since the relative K\"unneth formula concerns derived pushforward with compact support, this argument provides the mixed Hodge structures on $H^*_c(X_{a,b},\BQ)$. However, when $ord_2(a)=ord_2(b)$, where $ord_2$ is the $2$-adic valuation of an integer, the variety $X_{a,b}$ is singular. In that case, the mixed Hodge structure on $H^*_c(X,\BQ)$ is not directly related to that on $H^*(X,\BQ)$. To bridge this gap, we construct a minimal resolution $\pi:\widetilde{X}_{a,b}\to X_{a,b}$ and use $\widetilde{X}$ together with the Poincar\'e duality to relate $H^*_c(X_{a,b},\BQ)$ and $H^*(X_{a,b},\BQ)$. We also invoke the decomposition theorem for mixed Hodge modules to the resolution $\pi$ to determine the mixed Hodge structure on $IH^*(X_{a,b})$.  Our overall strategy is to employ mixed Hodge modules to compute the mixed Hodge structures of various cohomology groups, following the diagram below.
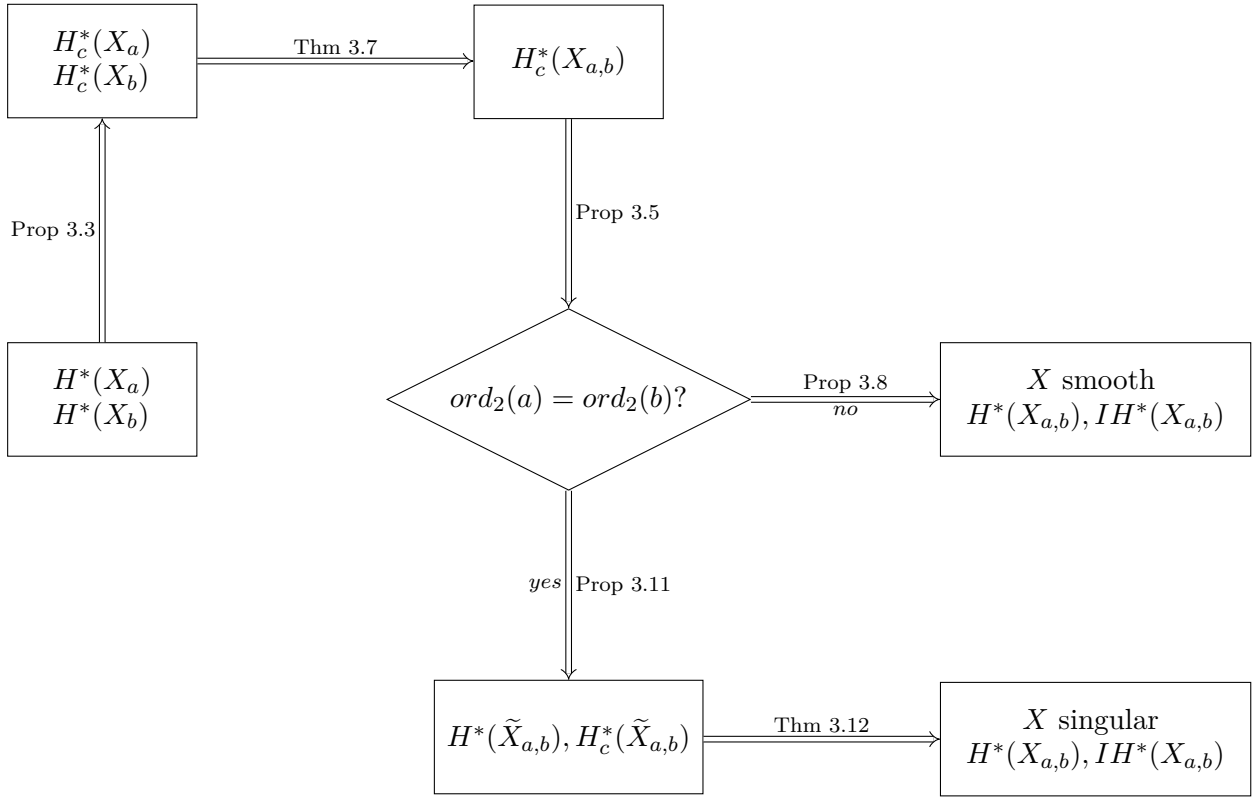
\[
\begin{tikzcd}
[cells={nodes={
        draw,           
        minimum width=2.5cm,
        minimum height=1.5cm,
    }},
    column sep=2.5cm,
    row sep=2.5cm]
\begin{array}{l}
H^*_c(X_a)\\
H^*_c(X_b)
\end{array}
\arrow[r,Rightarrow,"\textup{Thm }\ref{cpt}"] & H^*_c(X_{a,b}) \arrow[d,Rightarrow, "\textup{Prop }\ref{sing}"]& \\
\begin{array}{l}
H^*(X_a)\\
H^*(X_b)
\end{array}\arrow[u,Rightarrow,"\textup{Prop }\ref{3.3}"]& |[diamond, aspect=2]|{ord_2(a)=ord_2(b)?}\arrow[r,Rightarrow, "no"',"\textup{Prop }\ref{sm}"]\arrow[d,Rightarrow,"yes"',"\textup{Prop }\ref{res}"]& \begin{array}{c}X\textup{ smooth }\\H^*(X_{a,b}), IH^*(X_{a,b})\end{array}\\
&H^*(\widetilde{X}_{a,b}),H^*_c(\widetilde{X}_{a,b})\arrow[r,Rightarrow,"\textup{Thm }\ref{s}"] & \begin{array}{c}X\textup{ singular }\\H^*(X_{a,b}), IH^*(X_{a,b})
\end{array}
\end{tikzcd}
\]

A mixed Hodge structure $(V,W_\bullet,F^\bullet)$ is said to be of \emph{mixed Hodge-Tate type} if its $(p,q)$-component satisfies
\[
V^{p,q}=\textup{Gr}^p_F\textup{Gr}_{p+q}^WV=0 \textup{ unless }p=q.
\]
In particular, if $V$ is of mixed Hodge-Tate type, then $\textup{Gr}_{2k+1}^WV=0$ for every integer $k$,
and the mixed Hodge structure is completely determined by the dimensions of the graded pieces of the weight filtration. The mixed Hodge structures on $H^*(X_{a,b})$ and $IH^*(X_{a,b})$ are described follows.

\begin{thm}[Proposition \ref{sm}, Theorem \ref{s}]
Let $X_{a,b}$ be the cluster variety defined by \eqref{dfn1}. Then the mixed Hodge structures on $H^*(X_{a,b},\BQ)$ and $IH^*(X_{a,b},\BQ)$ are of mixed Hodge-Tate type. If $ord_2(a)\neq ord_2(b)$, then $X_{a,b}$ is smooth, so $H^*(X_{a,b},\BQ)=IH^*(X_{a,b},\BQ)$. The dimensions of the graded pieces of the weight filtration in this case are listed in the following table.

\begin{center}
\begin{tabular}{c|cccc}
$X_{a,b}$ & $W_0$ & $\textup{Gr}_2^W$&$\textup{Gr}_4^W$&$\textup{Gr}_6^W$ \\
\hline
$H^0$ & $1$ &   &   &  \\
$H^1$ &   & $1$ &   &  \\
$H^2$ &   & $a+b-2$ & $2$ &  \\
$H^3$ &   &    &  $a+b-1$ & $1$\\
\end{tabular}
\end{center}
If $ord_2(a)= ord_2(b)$, then $X_{a,b}$ is singular. The graded pieces of the weight filtration on $H^*(X_{a,b},\BQ)$ and $IH^*(X_{a,b},\BQ)$ are listed in the following tables.
\begin{center}
\begin{tabular}{c|cccc}
$X_{a,b}$ & $W_0$ & $\textup{Gr}_2^W$&$\textup{Gr}_4^W$&$\textup{Gr}_6^W$ \\
\hline
$H^0$ & $1$ &   &   &  \\
$H^1$ &   & $1$ &   &  \\
$H^2$ &   & $a+b-2$ & $2$ &  \\
$H^3$ &   &    &  $a+b-(a,b)-1$ & $1$\\
\end{tabular}

\begin{tabular}{c|cccc}
$X_{a,b}$ & $W_0$ & $\textup{Gr}_2^W$&$\textup{Gr}_4^W$&$\textup{Gr}_6^W$ \\
\hline
$IH^0$ & $1$ &   &   &  \\
$IH^1$ &   & $1$ &   &  \\
$IH^2$ &   & $a+b+(a,b)-2$ & $2$ &  \\
$IH^3$ &   &    &  $a+b-(a,b)-1$ & $1$\\
\end{tabular}
\end{center}
\end{thm}

\subsubsection{P side} 
Since the real analytic morphism $h:X\to \BR^3$ defined in \eqref{dfn2} is not complex algebraic, the BBDG decomposition cannot be applied directly to obtain the perverse decomposition. Instead, we view $h$ as a stratified map between stratified spaces and study the explicit geometry on each stratum. Let 
\[
U=\BR^3\setminus\{(u,v,w)\in\BR^3\mid uv=w=0\}
\]
be the open set in $\BR^3$ obtained by removing the $u$-axis and the $v$-axis. Denote by $V_{u^+},V_{u^-},V_{v^+},V_{v^-}$ be the strictly positive and negative half axes of the $u$-axis and $v$-axis, respectively, and let $O$ be the origin. Then 
\[
\BR^3=U\sqcup V_{u^+}\sqcup V_{u^-}\sqcup V_{v^+}\sqcup V_{v^-}\sqcup O
\]
is a Whitney stratification such that $h$ is locally trivial on each stratum. The geometry of $h$ with respect to the stratification is summarized as follows.

\begin{thm}[Proposition \ref{5}, Theorem \ref{5.3}]
Let $h:X_{a,b}\to \BR^3$, $U$, $V_{u^\pm}$, $V_{v^\pm}$ and $O$ be as above. Then:
\begin{enumerate}
\item Each fiber over $U$ is a $3$-torus $T^3$. The fundamental group of $U$ is free of rank $3$, and with a suitable choice of the basis, the monodromy action on $H^1$ of the fibers is given by matrices
\begin{equation*}
\begin{pmatrix}
1& 0& 0\\
0& 1& 0\\
0& b & 1
\end{pmatrix}
\quad\begin{pmatrix}
1& 0& 0\\
0& 1& 0\\
a& 0 & 1
\end{pmatrix}
\quad
\begin{pmatrix}
1& 0& 0\\
0& 1& 0\\
-a& 0& 1
\end{pmatrix}.
\end{equation*}
\item Each fiber over $V_{u^\pm}$ is a $3$-torus with $b$ $2$-tori collapsed to $b$ singular circles. Each fiber over $V_{v^\pm}$ is a $3$-torus with $a$ $2$-tori collapsed to $a$ singular circles. 
\item The fiber over $O$ satisfies:
 \begin{enumerate}
        \item If $ord_2(a)\neq ord_2(b)$, then $h^{-1}(O)$ is a $3$-torus with $a+b$ $2$-tori collapsed to $a+b$ singular circles.
        \item If $ord_2(a)=ord_2(b)$, then $h^{-1}(O)$ is a 3-torus with $a+b-\gcd(a,b)$ $2$-tori collapsed to $a+b-(a,b)$ singular circles and $\gcd(a,b)$ $2$-tori collapsed to $\gcd(a,b)$ singular points.   
 \end{enumerate}
\end{enumerate}
\end{thm}

Based on this geometric description of the fibers, we study the specialization map from the cohomology of the singular fiber to that of nearby smooth fibers, and use the resulting specialization data to compute the higher derived push-forward $R^ih_*\BQ_X$. Although we don't know whether $Rh_*\BQ_X\cong \oplus_i R^ih_*\BQ_X[-i]$, we prove that the Leray spectral sequence degenerates at $E_2$-page. Furthermore, we also decompose each $R^ih_*\BQ_X$ as a direct sum of sheaves supported on different strata. We then use the octahedral axiom to reorganize the direct summands according to the dimension of their supports, thereby producing the perverse truncation. 

\begin{thm} [Theorem \ref{pervtr}, Proposition \ref{perv split}]
The perverse cohomology sheaves are
\[
^\Fp\CH^k(Rh_*\BQ_X)=\begin{cases}
\BQ_{\BR^3}[1] & k=1,\\
j_*L[1]\oplus\BQ_u^{b-1}\oplus\BQ_v^{a-1} &k=2,\\
j_*\Lambda^2L[1]\oplus H& k=3,\\
\BQ_{\BR^3}[1] & k=4, \\
0 & \textup{otherwise},
\end{cases}
\]
and the perverse Leray spectral sequence
\[
   ^\Fp E_{2}^{p,q}=\BH^{p}(\BR^3,{^\Fp\CH}^q(Rh_*\BQ_X))\Longrightarrow H^{p+q}(X,\BQ)
   \]
   degenerates at $E_2$-page.
\end{thm}

\begin{thm}[Theorem \ref{PH}, \ref{PIH}]
Let $X_{a,b}$ be the cluster variety and be the real analytic proper map defined in \ref{defn2}. If $X_{a,b}$ is smooth, i.e. $ord_2(a)\neq ord_2(b)$, then the dimensions of the graded pieces of perverse filtration associated with $h:X\to\BR^3$ are listed in the following table.
\begin{center}
\begin{tabular}{c|cccc}
$X_{a,b}$ & $P_0$ & $\textup{Gr}_1^P$ & $\textup{Gr}_2^P$ & $\textup{Gr}_3^P$ \\    
\hline
$H^0$ & $1$ &   &   &  \\
$H^1$ &   & $1$ &   &   \\
$H^2$ &   & $a+b-2$& $2$ &  \\
$H^3$ &   &   & $a+b-1$ & $1$
\end{tabular}
\end{center}
If $X_{a,b}$ is singular, i.e. $ord_2(a)= ord_2(b)$, then the dimensions of the graded pieces of the perverse filtration associated with $h:X\to\BR^3$ are listed in the following table.
\begin{center}
\begin{tabular}{c|cccc}
$X_{a,b}$ & $P_0$ & $\textup{Gr}_1^P$ & $\textup{Gr}_2^P$ & $\textup{Gr}_3^P$ \\    
\hline
$H^0$ & $1$ &   &   &  \\
$H^1$ &   & $1$ &   &   \\
$H^2$ &   & $a+b-2$& $2$ &  \\
$H^3$ &   &   & $a+b-(a,b)-1$ & $1$
\end{tabular}
\end{center}
In this situation, the dimensions of the graded pieces of the perverse filtration on $IH^k(X)$ associated with $h$ are listed in the following table. 
\begin{center}
\begin{tabular}{c|cccc}
$X_{a,b}$ & $P_0$ & $\textup{Gr}_1^P$ & $\textup{Gr}_2^P$ & $\textup{Gr}_3^P$ \\    
\hline
$IH^0$ & $1$ &   &   &  \\
$IH^1$ &   & $1$ &   &   \\
$IH^2$ &   & $a+b+(a,b)-2$& $2$ &  \\
$IH^3$ &   &   & $a+b-(a,b)-1$ & $1$
\end{tabular}
\end{center}
\end{thm}

\subsection{Outline}
In Section 2, we review necessary background material. We recall the bounded derived category of constructible sheaves on real analytic varieties and the corresponding perverse $t$-structure with respect to the upper middle perversity function, as well as the theory of mixed Hodge modules. We also summarize the known results on the $P=W$ identity for cluster varieties, together with the canonically defined real analytic Lagrangian fibrations. At the end of the section, we prove a combinatorial lemma that will be used later to study the degenerations of singular fibers of the Lagrangian fibration. 

In Section 3, we compute the mixed Hodge structures on the cohomology, compactly supported cohomology and intersection cohomology of the 3-dimensional isolated cluster varieties of non-full rank. To this end, we apply the relative K\"unneth formula to a natural non-proper algebraic fibration of 2-dimensional isolated cluster varieties to compute the compactly supported cohomology of 3-dimensional isolated cluster varieties of non-full rank. We also use resolution of singularities and Verdier duality to determine the mixed Hodge structures on the ordinary cohomology and intersection cohomology.

In Section 4, we study the geometry of the Lagrangian fibrations for the 3-dimensional isolated cluster varieties of non-full rank. We compute the monodromy of the smooth fibers and describe the degenerations of the singular fibers. We then compute the higher direct image sheaves and prove that the associated Leray spectral sequence degenerates at the $E_2$-page. Using this information, we construct the perverse truncation and compute the perverse filtration associated with the Lagrangian fibrations.

In Section 5, we prove the $P=W$ and $PI=WI$ identities for 3-dimensional isolated cluster varieties. 

\subsection{Acknowledgements} 
I thank Mark de Cataldo, Thomas Lam, Ian Le, Christian Schnell, Junliang Shen, David Speyer, Ruijie Yang, Qizheng Yin, Dingxin Zhang, and Xiping Zhang for helpful discussions. I am partially supported by NSFC grant 12571045 and the Fundamental Research Funds for the Central Universities.

\section{Preparatory Results}

\subsection{Derived category}
In this section, we recall some facts about the derived category of sheaves on real analytic varieties. Coefficients are always assumed to be rational. Most of the results have counterparts in the theory of  mixed Hodge modules when the underlying spaces are complex algebraic varieties.

Let $X$ be a real analytic variety and let $D^b_c(X)$ denote the bounded derived category of $\BR$-constructible sheaves on $X$; see \cite[Chapter VIII]{KS}. The category $D^b_c(X)$ is a triangulated category, and in particular, satisfies the octahedral axiom.

\begin{prop} \label{T4}
Let $F,G,H\in D^b_c(X)$ and $F\xrightarrow{a} G\xrightarrow{b} H$ be two morphisms in $D^b_c(X)$. Suppose we are given three distinguished triangles
\begin{equation} \label{3triangle}
\begin{split}
F\xrightarrow{a} G&\to X\to,\\
G\xrightarrow{b} H&\to Z\to,\\
F\xrightarrow{b\circ a} H&\to Y\to.
\end{split}
\end{equation}
Then there exists a natural distinguished triangle
\[
X\to Y\to Z\to
\]
such that the following diagram commutes
\begin{equation} \label{3tri}
\begin{tikzcd}
   & & & Z\\
   F\arrow[rd,"a"]\arrow[rr,"b\circ a"]& & H\arrow[r]\arrow[ru] & Y\arrow[u,dotted] \\
    & G\arrow[ru,"b"]\arrow[rrd]& & \\
          & & & X.\arrow[uu,dotted]
\end{tikzcd}
\end{equation}
\end{prop}

As a convention, we will only draw a diagram of the form \eqref{3tri} when we intend to apply the octahedral axiom to \eqref{3triangle}. The following well-known result follows from \cite[Proposition III.8.1]{Har} and the proper base change theorem.

\begin{lem}\label{push}
Let $f:X\to Y$ be a continuous map between topological spaces and let $F$ be a sheaf on $X$. Then $R^kf_*F$ is the sheafification of the presheaf
\[
U\mapsto H^k(f^{-1}(U),F).
\]
When $f$ is proper, we have 
\[
(R^kf_*F)_p=H^k(f^{-1}(p),F).
\]
\end{lem}

\begin{prop}
Let $X$ be a real analytic space, let $j:U\to X$ be an open embedding, and let $i:Z\to X$ be its closed complement. Then for any complex $F\in D^b_c(X)$, there exist distinguished triangles
\[
i_*i^!F\to F\to Rj_*j^*F\to
\]
and
\[
j_!j^*F\to F\to i_*i^*F\to.
\]
\end{prop}

Analogously to the sheaf-theoretic definition of the cohomology functors, compactly supported cohomology can also be formulated in derived categories. Indeed, the functors $H^k_c(X,-):D^b_c(X)\to Ab$ are defined as the higher derived functors of the compactly supported global section functor, \emph{i.e.} $H^k_c(X,F)=R^k\Gamma_c(X,F)$. Moreover, for a continuous map $f:X\to Y$, we have $H^k_c(X,F)=H^k_c(Y,j_!F)$. The compactly supported cohomology of topological spaces are exactly the compactly supported sheaf cohomology of the constant sheaves.

\subsection{Mixed Hodge modules}
In this section, we briefly recall some basic facts concerning mixed Hodge modules. Standard references are \cite{saito,saito1}. Let $X$ be a complex algebraic variety and let $MHM(X)$ be the category of mixed Hodge modules on $X$. We follow the convention in $[1]$ for the shift functor and denote the Tate twist$(1)$ for the Tate twist. The following are some geometric results of mixed Hodge modules. Below we collect some standard results about mixed Hodge modules; proofs are omitted.

\begin{thm}\label{MHMdual}
\begin{enumerate}
\item If $X$ is a smooth algebraic variety of dimension $n$, then the dualizing complex in $D^b(MHM(X))$ is given by $\omega_X=\BQ_X[2n](n)$.
\item For any complex algebraic variety $X$, the cohomology $H^*(X)$, compactly supported cohomology $H^*_c(X)$ and intersection cohomology $IH^*(X)$ all carry natural mixed Hodge structures. Moreover, if $X$ is smooth and connected, the natural paring 
\[
H^k(X)\times H^{2\dim_\BC X-k}_c(X)\to H_c^{2\dim_\BC X}(X)\cong \BQ(-\dim_\BC X)
\]
is perfect and compatible with the mixed Hodge structures, inducing an isomorphism 
\[
H^k(X)\cong H^{2\dim_\BC X-k}_c(X)^\vee(-\dim_\BC X)
\]
as mixed Hodge structures.
\item For a morphism $f:X\to Y$ between smooth algebraic varieties, each higher direct image $R^kf_*\BQ_X$ carries a natural mixed Hodge module structure. Furthermore, if $U\subset Y$ is an open subset over which $f$ is smooth, then $(R^kf_*\BQ_X)|_U$ is the variation of mixed Hodge structures associated with the $k$-th cohomology of the fibers.  
\item Let $i:Z\to X$ be a closed embedding of complex algebraic varieties. Then for any $G\in MHM(X)$, we have $\BQ_Z\otimes G=i_*i^*G$.
\end{enumerate}
\end{thm}

We conclude this section by recalling some facts about the six-functor formalism of mixed Hodge modules.

\begin{thm}\label{MHM2}
    \begin{enumerate}
        \item Let $F\in D^b(MHM(X))$ and let $f:X\to Y$ be a morphism of algebraic varieties. Then the direct image functors satisfy $Rf_!F,Rf_*F\in D^b(MHM(Y))$, and their cohomology $R^kf_!F,R^kf_*F\in MHM(Y)$ for each $k$. Moreover, there is a canonical duality isomorphism $\BD_YRf_*F\cong Rf_!\BD_X F$.
        \item Let $i:Z\to X$ be a closed embedding of algebraic varieties, with $j:U\to X$ be the complementary open embedding. For any $F\in D^b(MHM(X))$, there are distinguished triangles 
        \[
        i_*i^!F\to F\to Rj_*j^*F\to
        \]
        and
        \[
        j_!j^*F\to F\to i_*i^*F\to
        \]
        in $D^b(MHM(X))$.
        \item Let $S$ be a complex algebraic variety and let $f_i:X_i\to Y_i$, $i=1,2$ be morphism of $S$-varieties. Denote by  
        \[
         f=f_1\times_Sf_2:X_1\times_S X_2\to Y_1\times_SY_2
        \]
         the induced fibered product morphism. Then for any $G_i\in D^b_c(MHM(X_i))$, there is a natural isomorphism in $D^b_c(MHM(Y_1\times_SY_2))$
        \[
        Rf_{1!}G_1\overset{L}{\boxtimes}_SRf_{2!}G_2\xrightarrow{\sim}Rf_!(G_1\overset{L}{\boxtimes}_SG_2),
        \]
        where $\boxtimes_S$ denotes the external tensor product over $S$.
        \end{enumerate}
\end{thm}

\subsection{Perverse filtrations on real analytic spaces}
Since the theory of perverse sheaves in non-complex setting is not standard, we include relevant backgrounds in the section, following \cite[Chapter X]{KS}. Throughout this section, dimension refers to dimension of real analytic spaces unless otherwise stated. To define the perverse $t$-structure on real analytic manifolds, one must first introduce the notion of perversity functions.

\begin{defn}{\cite[Section 10.2]{KS}}
   A perversity function is an integral-valued function $p:\BZ\to\BZ$ satisfying  $p(n)-p(n+1)\in\{0,1\}$ for all $n\in\BZ$. 
\end{defn}

For a given perversity function $p$ and any real analytic manifold $X$, there exists a natural perverse $t$-structure $({^\Fp D^{\le0}(X)},{^\Fp D^{\ge0}}(X))$ on the derived category of $\BR$-constructible sheaves $D^b_c(X)$ defined by the following.

\begin{defn}{\cite[Proposition 10.4]{KS}} \label{perv}
    Let $p$ be a perversity function and let $X$ be a real analytic manifold. Suppose $X=\sqcup_\alpha X_\alpha$ is a subanalytic stratification into equidimensional strata, and let $i_\alpha:X_\alpha\to X$ denote the locally closed embedding. For $F\in D^b_c(X)$, assuming that $i_\alpha^*F$ and $i^!_\alpha F$ have locally constant cohomology sheaves for all $\alpha$, then
    \begin{enumerate}
       \item $F\in {^\Fp D^{\le0}(X)}$ if and only if $\CH^j(i_\alpha^*F)=0$ for any $j>p(\dim X_\alpha)$.
       \item $F\in {^\Fp D^{\ge0}(X)}$ if and only if $\CH^j(i_\alpha^!F)=0$ for any $j<p(\dim X_\alpha)$.
    \end{enumerate}
\end{defn}

The heart of this $t$-structure is the category of perverse sheaves $Perv(X)={^\Fp D^{\le0}(X)}\cap{^\Fp D^{\ge0}}(X)$. The $t$-structure on $D^b_c(X)$ induces perverse truncation functors $^\Fp\tau_{\le n}:D^b_c(X)\to {^\Fp D^{\le n}(X)}$, which are left adjoint to the inclusion ${^\Fp D^{\le n}(X)}\to D^b_c(X)$. Equivalently, for a give $F$, the truncation $^\Fp\tau_{\le n}F$ can be identified as $F'$ whenever there exists distinguished triangle
\[
F'\to F\to F''\to
\]
with $F'\in D^{\le n}(X)$ and $F''\in D^{\ge n+1}(X)$. The truncation functor from below $^\Fp\tau_{\ge n}$ is defined analogously. The perverse cohomology sheaves are defined as $^\Fp\CH^n(F)=({^\Fp}\tau_{\ge n}{^\Fp}\tau_{\le n}F)[n]$.

\begin{defn} \label{2.1}
   Let $p$ be a perversity function. The perverse filtration of a complex $F\in D^b_c(X)$ is the increasing filtration on the hypercohomology $\BH^*(X,F)$ defined by
   \[
   P_k\BH^*(X,F):=\textup{Im}\{\BH^*(X,{{^\Fp}\tau_{\le k}}F)\to \BH^*(X,F)\},\,k\in\BZ.
   \]
   The perverse spectral sequence is the $E_2$-spectral sequence 
   \[
   ^\Fp E_2^{p,q}=H^p\left(X, {^\Fp\CH^q}(F)\right)\Rightarrow H^{p+q}(X,F).
   \]
   The (increasing) filtration induced by this spectral sequence is exactly the perverse filtration.  
\end{defn}

The following characterization of a perverse truncation will be used later.
\begin{prop} \label{byhand}
Let $F\in D^b_c(X)$. Suppose that for each $k\in\BZ$, there exist $G_k\in D^b_c(X)$ and morphism $\iota_k:G_k\to G_{k+1}$ satisfying:
\begin{enumerate}
\item The cones $Cone(\iota_k)$ are $k$-shifted perverse sheaves, i.e.
\[
Cone(\iota_k)\in Perv(X)[-k]\subset D^b_c(X).
\]
\item $G_k=0$ for $k\ll0$ and $G_k=F$ for $k\gg0$.
\end{enumerate}
Then the perverse truncation satisfy $\tau_{\le k} F\cong G_k$ and the perverse cohomology sheaves are $^\Fp\CH^k(F)=Cone(\iota_k)[k]$. Moreover, the perverse filtration is recovered as
\[
P_k\BH^*(X,F)=\textup{Im}\{\BH^*(X,G_k)\to \BH^*(X,F)\}.
\]
If, in addition, the perverse spectral sequence is $E_2$-degenerate, then there is an isomorphism 
\begin{equation} \label{nonsplit}
P_k\BH^*(X,F)\cong\bigoplus_{i\le k} \BH^*(X,{^\Fp}\CH^i(F)).
\end{equation}
\end{prop}

\begin{proof}
   To prove $G_k\cong {^\Fp\tau_{\le k}}F$, it suffices to show that $G_k\in {^\Fp}D^{\le k}(X)$ and  $Cone(G_k\to F)\in {^\Fp}D^{\ge k+1}(X)$. In fact, since $Cone(\iota_i)\in {^\Fp}D^{\le k}(X)$ for $i\le k$, their successive extension $G_k$  also lies in ${^\Fp}D^{\le k}(X)$ by \cite[Lemma 8.1.7]{HTT}. The assertion for $Cone(G_k\to F)$ follows from a similar argument. In $E_2$-degenerate case, the direct sum decomposition is a formal consequence of the spectral sequence of filtered complex (with increasing filtration).
\end{proof}

We often fix a perversity function rather than considering all possible perversity function. In the context of complex algebraic varieties, the standard choice is the middle perversity function $p=-n/2$. However, since the real analytic setting involves odd dimensions, we adopt throughout this paper the upper middle perversity function $p_+(n)=\ceil{-n/2}$.

\begin{defn}
   Let $f:X\to Y$ be a proper map of real analytic varieties with equidimensional fibers. 
   The perverse filtration $P_kH^*(X,\BQ),\,{k\ge0}$ associated with $f$ is defined as the shift of the filtration $P_k\BH^*(Y,Rf_*\BQ_X)$ such that $1\in P_0H^0(X,\BQ)$.
\end{defn}

\begin{example}
We illustrate perverse sheaves and perverse filtrations in the real analytic setting with the following examples.
\begin{enumerate}
\item 
   Let 
   \[
   O=V_0\subset V_1\subset V_2\subset V_3=\BR^3
   \]
   be a complete flag of real linear subspaces in $\BR^3$. Then with respect to the upper middle perversity function, the complexes
   \[
   \BQ_{V_0},\BQ_{V_1},\BQ_{V_2}[1],\BQ_{V_3}[1]
   \]
   are perverse sheaves. By contrast, with respect to the lower middle perversity function $p(n)=\floor{-n/2}$, the perverse objects are 
   \[
   \BQ_{V_0},\BQ_{V_1}[1],\BQ_{V_2}[1],\BQ_{V_3}[2].
   \]
   Roughly speaking, the difference of the two approximations of the middle perversity occurs precisely on odd-dimensional strata.
\item 
   Let $p:X=\BC^*\to \BR^1$ be the projection given by $p(z)=\log|z|$. There is a natural decomposition
   \[
   Rp_*\BQ_{\BC^*}=\BQ_{\BR^1}\oplus\BQ_{\BR^1}[-1].
   \]
   By definition, the sheaf $\BQ_{\BR^1}$ is perverse with respect to the upper middle perversity. So the perverse truncation is given by
   \[
   ^\Fp\tau_{\le k} Rp_*\BQ_{\BC^*}=\begin{cases}
       0, & k\le -1,\\
       \BQ_{\BR^1}, & k=0,\\
       \BQ_{\BR^1}\oplus\BQ_{\BR^1}[-1],& k\ge1.
   \end{cases}
   \]
   Hence the dimensions of the graded pieces of the perverse filtration associated with $p$ can be listed in the following table.
   \begin{center}
\begin{tabular}{c|cc}
$\BC^*$ & $P_0$ &$\textup{Gr}_1^P$ \\
\hline
$H^0$ & $1$ &    \\
$H^1$ &   & $1$  
\end{tabular}
\end{center}
   If one instead adopts the lower middle perversity function, then $\BQ_X[1]$ becomes perverse and all perverse truncation indices are shifted downward by $1$. After the global shifting which guarantees $1\in P_0H^0(\BC^*,\BQ)$, the resulting perverse filtration coincides with the one defined using the upper middle perversity.
\item
   We now present an example where the perverse spectral sequence fails to be $E_2$-degenerate. Let $f:S^3\to S^2$ be the Hopf fibration. Since $S^2$ is simply-connected, we have 
   \[
   R^kf_*\BQ_{S^3}=\BQ_{S^2} \textup{for } k=0,1.
   \]
   From the total cohomology of $X$ we see that 
   \[
   Rf_*\BQ_X\neq R^0f_*\BQ_X\oplus R^1f_*\BQ_X[-1]
   \]
   and therefore the perverse spectral sequence is not $E_2$-degenerate. 
   
   Since $\BQ_{S^2}[1]$ is perverse, we have
   \[
   ^\Fp\tau_{\le k} Rp_*\BQ_{S^3}=\begin{cases}
       0 & k\le 0,\\
       \BQ_{S^2} & k=1,\\
       Rf_*\BQ_{S^3}& k\ge1,
   \end{cases}
   \]
   Then the cohomological perverse filtration associated with $f$ in the sense of Definition \ref{2.1}.(2) is 
   \[
   P_k\BH^*(S^3,Rf_*\BQ)=\begin{cases}
   0 & k\le 0\\
   \text{Im}\{H^*(S^2,\BQ)\to H^*(S^3,\BQ)\}=\BQ\cdot1 &k=1\\
   H^*(S^3,\BQ) & k\ge2.
   \end{cases}
   \]
   So after shifting the index of the perverse filtration such that $1\in P_0H^0(S^3)$, the graded pieces of the perverse filtration is listed in the following table.
   \begin{center}
\begin{tabular}{c|cc}
$S^3$ & $P_0$ &$\textup{Gr}_1^P$ \\
\hline
$H^0$ & $1$ &    \\
$H^1$ &   &    \\
$H^2$ &   &    \\
$H^3$ &   & $1$  
\end{tabular}
\end{center}
It is immediate that the decomposition \eqref{nonsplit} does not hold without $E_2$-degeneracy hypothesis.
\end{enumerate}
\end{example}

\subsection{P=W for cluster varieties}
In \cite{Z,Z2}, the $P=W$ phenomenon is also observed on a special class of varieties known as cluster varieties. These are affine varieties defined as the spectra of certain commutative algebras associated with decorated directed graphs and combinatorial operations; see \cite{FZ}. Under mild finite-type hypothesis, such as acyclicity or the Louise property \cite{LS},  cluster varieties admit coverings by isolated cluster varieties, which are constructed from matrices as follows. 

For an $m\times n$ integer matrix $M=(a_{ij})$, the associated cluster variety, denoted as $X(M)$, is the $m+n$ dimensional variety defined by the equations
\[
x_jx_j'=\prod_{i=1}^m z_i^{a_{ij}}+1\textup{ for }1\le j\le n \textup{ and } z_1,\cdots,z_m\ne 0.
\]
inside the space
\[
\BC^{2n+m}=\Spec\,\BC[x_1,\cdots,x_n,x_1',\cdots,x_n',z_1,\cdots,z_m].
\]

Isolated cluster varieties play a crucial role in the topological study of cluster varieties. Indeed, an  iterative Mayer-Vietoris argument starting from isolated cluster varieties yields an effective algorithm for computing cohomology groups of all cluster varieties of Louise type; see \cite{LS,Z}. Moreover, each isolated cluster variety admits a natural torus fibration $h:X(M)\to \BR^{n+m}$  \footnote{There is a typo in the construction of $h$ in \cite{Z2}.}:
\begin{equation} \label{h}
\begin{split}
  h(x_1,\cdots,x_n,x_1'&,\cdots,x_n',z_1,\cdots,z_m)\\
  =&(|x_1|^2-|x_1'|^2,\cdots,|x_n|^2-|x_n'|^2,\log|z_1|,\cdots,\log|z_m|).
\end{split}
\end{equation}

When $M$ is of full column rank, the isolated cluster variety $X(M)$ is smooth and can be realized as a finite \'etale quotient of a product of cluster varieties of dimension 1 or 2. In this case, the map $h$ can also be described explicitly. For such full-rank isolated cluster varieties, the $P=W$ identity holds.
    
\begin{thm}{\cite[Theorem 1.3]{Z2}}
    Let $M$ be an $m\times n$ integer matrix with full column rank. Let $X(M)$ be the associated isolated cluster variety and $h:X(M)\to \BR^{n+m}$ be the associated Lagrangian fibration. Then the $P=W$ identity holds on the intersection cohomology, i.e.
    \[
    P_kH^*(X(M),\BQ)=W_{2k} H^*(X(M),\BQ)=W_{2k+1}H^*(X(M),\BQ),~~ k\ge 0, 
    \]
    where the perverse filtration is defined by the map $h$.
\end{thm}

When the matrix $M$ is not of full column rank, the behavior of the cluster varieties are very different. For example, $X(M)$ may not be smooth, and it is no longer a quotient of a product of lower-dimensional cluster varieties. Moreover, the mixed Hodge structure and the perverse filtrations become far more intricate. Even in the simplest non-full rank case $M=(1,1)$, the computation of its mixed Hodge structure is already nontrivial; see \cite[Section 4]{ZZ}. 

In this paper, we focus on the $3$-dimensional isolated cluster varieties $X(M)$ where $M$ is non-full rank. By the classification of 3-dimensional cluster varieties in \cite{ZZ}, aside from the trivial cases $\begin{pmatrix}0\\0\end{pmatrix}$ and $(a,0)$ whose cluster varieties are already studied in the full rank setting, the only remaining possibility is that $M$ is a $1\times 2$ matrix $(a,b)$ with $a,b\neq0$. Explicitly, the corresponding cluster varieties, denoted by $X_{a,b}$, are
\[
\begin{cases}
xx'=z^a+1\\
yy'=z^b+1\\
\end{cases}~~ \textup{where }x,x',y,y',z\in\BC\textup{, and }z\neq0.
\]

In fact, we will see in Section 3 that for a certain choice of $a$ and $b$ the cluster variety $X_{a,b}$ is neither smooth nor a product of lower-dimensional cluster varieties. Nevertheless, it remains closely related to 2-dimensional cluster varieties in the following way. Let $X_a$ be the cluster variety $xx'=z^a+1$, $z\neq0$. Then there is a commutative diagram (not a fibered product)
\begin{equation}\label{a}
\begin{tikzcd}
X_a\arrow[r,"f_a"]\arrow[d,"h_a"] & \BC^*\arrow[d,"\log|\cdot|"]\\
\BR^2\arrow[r,"pr_2"] &\BR
\end{tikzcd}
\end{equation}
where $f_a(x,x',z)=z$ and $h$ is defined by \eqref{h}. Then the 3-dimensional cluster variety $X_{a,b}$ is the fibered product $X_{a,b}=X_a\times_{\BC^*} X_b$, or equivalently, we have the fibered diagram

\begin{equation}\label{a1}
\begin{tikzcd}
X_{a,b}\arrow[r,]\arrow[d] & X_a\arrow[d,"f_a"]\\
X_b\arrow[r,"f_b"] &\BC^*.
\end{tikzcd}
\end{equation}

From the definition, the map $h:X_{a,b}\to\BR^3$ is obtained as the composition
\[
X_a\times_{\BC^*} X_b\to X_a\times_\BR X_b\to \BR^2\times_\BR\BR^2=\BR^3.
\]
However, it is unclear to the author how to compute the mixed Hodge structures and the perverse filtration for a general fibered product directly. In the follow sections, we proceed by analyzing the explicit geometry of these 3-dimensional cluster varieties.

\subsection{A combinatorial lemma}
In this section, we prove a combinatorial lemma concerning certain marking and numbering on a circle and apply it to roots of $-1$. For two points $P,Q\in S^1$, we denote by $\overset{\frown}{PQ}$ the counter-clockwise arc from $P$ to $Q$.

Let $T=\{t_1,\dots,t_n\}$ be a finite set of distinct points on a circle, arranged in counter-clockwise order. Let 
\[
A_T=\{\overset\frown{t_1t_2},\dots, \overset\frown{t_{n-1}t_n},\overset\frown{t_{n}t_1}\}
\]
be the set of simple arcs whose endpoints are adjacent points in T. Let $V_T$ be the $\BQ$-vector space  of all maps $A_T\to \BQ$. Thus, an element of $V_T$ assigns a rational number to each simple arc. For any embedding $T'\subset T$, there is a natural inclusion $\iota:A_{T}\to A_{T'}$, sending each arc $I\in A_{T}$ to the unique arc in $A_{T'}$ that contains $I$. This inclusion induces a linear map $\iota_*: V_{T}\to V_{T'}$ which sends the map $l:A_{T}\to\BQ$ to 
\begin{equation} \label{comb1}
\iota_* (l)(I')=\sum_{J\subset I',J\in A_{T}} l(J),~~I'\in A_{T'}.
\end{equation}
We remark that $\iota_*$ is \emph{not} the map defined by the precomposition (pull-back) with $\iota$, but rather should be interpreted as the integration (push-forward) along $\iota$.
\begin{lem} \label{comblemma}
   Let $T_1,T_2$ be two non-empty finite sets of points on $S^1$. The natural inclusions $\iota_i:T_i\to T_1\cup T_2$ induce a linear map
   \[
   \begin{split}
   g:V_{T_1\cup T_2}&\to V_{T_1}\oplus V_{T_2}\\
   l&\mapsto \iota_{1*}(l)\oplus  \iota_{2*}(l).
   \end{split}
   \]
   Then the kernel of $g$ satisfies
   \[
    \ker g\cong\begin{cases}
        \BQ& \textup{if  }T_1\cap T_2=\varnothing,\\
        0& \textup{if }T_1\cap T_2\neq\varnothing.
    \end{cases}
   \]
\end{lem}

\begin{proof}
   Let $l:\Sigma_{T_1\cup T_2}\to\BQ$ be a map in $\ker g$. Then for $i=1,2$ and $P,Q\in T_i$, the arc $\overset{\frown}{PQ}$ is the union of simple arcs in $A_{T_i}$, and by \eqref{comb1} we have
   \begin{equation} \label{comb2}
   \sum _{J\subset \overset{\frown}{PQ}, J\in A_{T_1\cup T_2}}l(J)=0.
   \end{equation}
   Suppose first that $T_1\cap T_2\neq \varnothing$, and fix a point $a\in T_1\cup T_2$. For any arc $\overset{\frown}{pq}\in \Sigma_{T_1\cup T_2}$:
   \begin{enumerate}
    \item If $P,Q\in T_1$ or $P,Q\in T_2$, then $l(\overset{\frown}{pq})=0$ by \eqref{comb2}.
    \item If $p,q$ do not lie in the same $T_i$, without loss of generality we may assume $P\in T_1$ and $Q\in T_2$. Then $\overset{\frown}{AP}+\overset{\frown}{PQ}=\overset{\frown}{AQ}$. By \eqref{comb2}, we have 
    \[
   \sum _{J\subset \overset{\frown}{ap}, J\in A_{T_1\cup T_2}}l(J)=\sum _{J\subset \overset{\frown}{aq}, J\in A_{T_1\cup T_2}}l(J)=0
   \]
   and hence $l(\overset{\frown}{pq})=0$.
   \end{enumerate}
   Therefore $l=0$ and hence $\ker g=0$.

   Now suppose $T_1\cap T_2=\varnothing$. Call an arc $\overset{\frown}{PQ}\in A_{T_1\cup T_2}$ of \textit{type I} if $P\in T_1$ and $Q\in T_2$, and \textit{of type II} if $P\in T_2$ and $Q\in T_1$. We claim that there exists a constant $c\in\BQ$ such that $l$ takes value $c$ on every type I arcs, $-c$ on all type II arcs, and 0 on all other simple arcs. In fact, the interior of any arc $\overset{\frown}{PQ}\in A_{T_1}$ contains no points of $T_1$, so it is consists of exactly 1 type I arc and 1 type II arc, and possibly some other simple arcs in $A_{T_2}$. The same argument works for arcs in $A_{T_2}$. Hence the type I and type II arcs alternate around on the circle. Equation \eqref{comb2} then implies that the values of $l$ on adjacent type I and type II arcs sums to 0, proving the claim. Conversely, it is straightforward to verify that any $l$ of the form above lie in $\ker g$. We conclude that $\ker g\cong\BQ$ when $T_1\cap T_2=\varnothing$. 
\end{proof}

Let $V_T^0$ be the subspace of functions $l\in V_T$ satisfying $\sum_{I\in A_T}l(I)=0$. Then we have the following consequence of Lemma \ref{comblemma}

\begin{cor} \label{combcor}
    With notation as in Lemma \ref{comblemma}, the kernel of the induced map 
    \[
    g^0: V_{T_1\cup T_2}^0\to V_{T_1}^0\oplus V_{T_2}^0
    \]
    satisfies
    \[
    \ker g\cong\begin{cases}
        \BQ& \textup{if  }T_1\cap T_2=\varnothing,\\
        0& \textup{if }T_1\cap T_2\neq\varnothing.
    \end{cases}
   \]
\end{cor}

\begin{proof}
    We have a commutative diagram
    \[
    \begin{tikzcd}
    V_{T_1\cup T_2}^0\arrow[r,"g^0"]\arrow[d,hookrightarrow]& V_{T_1}^0\oplus V_{T_2}^0\arrow[d,hookrightarrow]\\
    V_{T_1\cup T_2}\arrow[r,"g"]& V_{T_1}\oplus V_{T_2}.
    \end{tikzcd}
    \]
   A diagram chase shows that $\ker g^0=\ker g$. The result then follows from Lemma \ref{comblemma}.
\end{proof}

We now apply the above combinatorial argument to the roots of $-1$ on the unit circle. For any integer $d>0$, let 
\[
Z_d=\{z\mid z^d=-1\}\subset\BC
\]
be the set of all $d$-th roots of $-1$. The following elementary observation will be used later.

\begin{lem} \label{int}
   Let $a,b$ be positive integers. Then 
   \[
   Z_a\cap Z_b=\begin{cases}
   Z_{(a,b)}, & \textup{if }ord_2(a)=ord_2(b),\\
   \varnothing & \textup{otherwise,}
   \end{cases}
   \] 
   where $(a,b)$ denotes the greatest common divisor and $ord_2$ is the $2$-adic valuation of an integer.
\end{lem}

\begin{proof}
   Suppose $e^{i\pi\psi}\in Z_a\cap Z_b$. Then $e^{i\pi a\psi}=e^{i\pi b\psi}=-1$, so both $a\psi$ and $b\psi$ are odd integers. By B\'ezout theorem, $(a,b)\psi$ is also an integer. Moreover, it divides  divides both $a\psi$ and $b\psi$. Therefore, $a/(a,b)$ and $b/(a,b)$ are both odd integers, which is equivalent to $ord_2(a)=ord_2(b)$. In this case, one checks directly that $Z_a\cap Z_b=Z_{(a,b)}$. 
\end{proof}

The following corollary is an  immediate consequence of Corollary \ref{combcor} and will be used in Section 4.2 to study the degeneration of the central fiber.

\begin{cor} \label{comb}
Let $a,b$ be positive integers. Then the kernel of the natural map
\[
g: V_{Z_a\cup Z_b}^0\to V_{Z_a}^0\oplus V_{Z_b}^0
\]
satisfies
    \[
    \ker g\cong\begin{cases}
        \BQ& \textup{if  }T_1\cap T_2=\varnothing,\\
        0& \textup{if }T_1\cap T_2\neq\varnothing.
    \end{cases}
   \]
\end{cor}

\section{W side}
\setcounter{subsection}{-1}
\subsection{Overview}
The goal of this section is to study the mixed Hodge structures of 3-dimensional isolated cluster varieties of non-full rank, \emph{i.e.} the affine varieties $X_{a,b}\subset \BC^5$ defined by equations
\[
\begin{cases}
xx'=z^a+1\\
yy'=z^b+1\\
\end{cases}~~ \textup{where }x,x',y,y',z\in\BC\textup{, and }z\neq0.
\]
Our approach proceeds in the following steps.
\begin{enumerate}
\item Recalculate the mixed Hodge structures of 2-dimensional isolated cluster varieties 
\[
X_d=\{xx'=z^d+1,~~z\neq0\}
\]
via mixed Hodge modules, by studying the decomposition associated with the natural non-proper morphism $X_d\to\BC^*$.
\item Apply the relative K\"unneth formula for Mixed Hodge modules to compute the mixed Hodge structure on the compactly supported cohomology $H^*_c(X_{a,b})$.
\item Resolve the singularities of $X_{a,b}$ and compute the mixed Hodge structures on $H^*(X_{a,b})$ and the intersection cohomology $IH^*(X_{a,b})$.
\end{enumerate}

\subsection{Mixed Hodge structure of $X_d$} 
The mixed Hodge structures on  $H^*(X_d)$ were previously computed in \cite[Section 7]{LS} by an analytic deformation retract argument and integrations of differential forms. However, this framework does not adapt well with the fibered products \eqref{a1} and hence does not generalize to the cases of interest here. For our purpose, we recalculate the mixed Hodge numbers of $X_d$ using mixed Hodge modules, by studying the decomposition associated with the natural morphism
\begin{equation} \label{f}
\begin{array}{c @{\;} c @{\;} c}
f_d: & X_d & \to \BC^* \\[4pt]
   & (x,x',z) & \mapsto z
\end{array}
\end{equation}
We remark that that the BBDG decomposition theorem cannot be applied directly to $f$ since it is not proper. 

Let $Z_d=\{\xi\in\BC^*\mid\xi^d=-1\}$ and $U_d$ be the open complement. Denote by $i_d:Z_d\to \BC^*$ and $j_d:U_d\to \BC^*$ be the corresponding immersions. Since the parameter $d$ is fixed throughout this section, we suppress it from the notation. By definition, the morphism $f$ is a trivial $\BC^*$ fibration over $U$. The fibers of $f$ are

\[
f^{-1}(p)=\begin{cases}
\BC^*, & p\in Z,\\
\BC\bigvee\BC,& p\in U,
\end{cases}
\]
where $\bigvee$ denotes the one-point union.

\begin{prop} \label{2d}
With the notation above, there exists an isomorphism
\begin{equation} \label{2d0001}
Rf_*\BQ_{X}=\BQ_{\BC^*}\oplus j_!\BQ_{U}[-1](-1)
\end{equation}
in $D^b(MHM(X))$.
\end{prop}

\begin{proof}
We first claim that 
\[
R^kf_*\BQ_X=
\begin{cases}
\BQ_{\BC^*}& k=0,\\
j_!\BQ_{U}(-1) & k=1,\\
0& \textup{otherwise}.
\end{cases}
\]
In fact, $R^0f_*\BQ_X=\BQ_{\BC^*}$ follows from the connectedness of fibers. Since $f$ trivializes over $U$ as the projection $:U\times \BC^*\to U$, the sheaf $R^1f_*\BQ_X|_U$ is a trivial variation of Hodge structure whose fiber is $H^1(\BC^*,\BQ)=\BQ(-1)$. Hence we obtain an isomorphism
\[
R^1f_*\BQ_{X}|_U=\BQ_U(-1).
\]
as mixed Hodge modules. Furthermore, for any point $p\in Z$ and sufficiently small Euclidean ball $V$ containing $p$, the preimage $f^{-1}(V)$ is contractible, so $(R^kf_*\BQ_{X})_p=0$ for $k\ge1$. Therefore $R^1f_*\BQ_{X}=j_!\BQ_U(-1)$ and $R^kf_*\BQ_{X}=0$ for $k\ge 2$.  

The distinguished triangle of the standard truncation functors $\tau_{\le0}\to \textup{id}\to \tau_{\ge1}\to$ yields a distinguished triangle of mixed Hodge modules
\begin{equation} \label{1233}
\BQ_{\BC^*}\to Rf_*\BQ_{X}\to j_!\BQ_U(-1)[-1]\to.
\end{equation}
The extension class in $D^b_c(X)$
\begin{equation} \label{1234}
\begin{split}
&\Ext^1_{D^b_c(X)}(j_!\BQ_U[-1],\BQ_{\BC^*})=\Hom_{D^b_c(X)}(j_!\BQ_U,\BQ_{\BC^*}[2])\\
=&\Hom_{D^b_c(X)}(\BQ_U,j^!\BQ_{\BC^*}[2])
=\Hom_{D^b_c(X)}(\BQ_U,\BQ_U[2])\\
=&H^2(U,\BQ)=0
\end{split}
\end{equation}
Since the forgetful functor $D^b(MHM(X))\to D^b_c(X)$ is faithful, the equation (\ref{1234}) implies
\[
\Ext^1_{D^b(MHM(X))}(j_!\BQ_U[-1](-1),\BQ_{\BC^*})=0
\]
and hence the distinguished triangle (\ref{1233}) splits. Consequently, we conclude that 
\begin{equation} \label{01}
Rf_*\BQ_{X}=\BQ_{\BC^*}\oplus j_!\BQ_{U}[-1](-1)
\end{equation}
is an isomorphism in the derived category of mixed Hodge modules.
\end{proof}

\begin{cor}
In the notations above,  the mixed Hodge structure on $H^*(X_d,\BQ)$ is of mixed Hodge-Tate type, with non-zero graded pieces of the weight filtration listed in the table below. 
\begin{center}
\begin{tabular}{c|ccccc}
$X_d$ & $W_0$ &$\textup{Gr}_1^W$& $\textup{Gr}_2^W$&$\textup{Gr}_3^W$&$\textup{Gr}_4^W$ \\
\hline
$H^0$ & $1$ &  &     &  &  \\
$H^1$ &     &  & $1$ &  &  \\
$H^2$ &     &  & $d-1$ &  & $1$
\end{tabular}
\end{center}
\end{cor}

\begin{proof}
The hypercohomology of the distinguished triangle
\[
j_!\BQ_U\to \BQ_{\BC^*}\to i_*\BQ_Z\to
\]
yields an exact sequence of mixed Hodge structures
\[
\begin{split}
0\to &H^0(\BC^*,j_!\BQ_U)\to H^0(\BC^*,\BQ)\xrightarrow{a} H^0(Z,\BQ)\\
\to& H^1(\BC^*,j_!\BQ_U)\to H^1(\BC^*,\BQ)\to 0.
\end{split}
\]
The map $a$ is injective because it is simply the restriction map. So we have $H^1(\BC^*,j_!\BQ_U)=\BQ^{d}$ with the weight filtration 
\[
W_0H^1(\BC^*,j_!\BQ_U)=W_1H^1(\BC^*,j_!\BQ_U)\cong \textup{coker}\,a\cong \BQ^{d-1}
\]
and $W_2=H^1$. Then apply Proposition \ref{2d}.
\end{proof}

Dually, we obtain the following results for compactly supported cohomology.
\begin{prop} \label{3.3}
In the notation above, we have
\[
Rf_!\BQ_{X}=\BQ_{\BC^*}[-2](-1)\oplus Rj_*\BQ_U[-1]
\]
and
\[
H^k_c(\BC^*,Rj_*\BQ_U)=\begin{cases}
\BQ^{d}, & k=1,\\
0, & \textup{otherwise.}
\end{cases}
\]
The mixed Hodge structure on $H^k_c(\BC^*,Rj_*\BQ_{U})$ is of mixed Hodge-Tate type, with the non-zero graded pieces of the weight filtration listed in the table below.
\begin{center}
\begin{tabular}{c|ccccc}
$Rj_*\BQ_U$ & $W_0$ &$\textup{Gr}_1^W$& $\textup{Gr}_2^W$ \\
\hline
$H^1_c$ & $1$ &  & $d-1$ 
\end{tabular}
\end{center}

Consequently, the mixed Hodge structure on $H^k_c(X_d,\BQ)$ is of mixed Hodge-Tate type, with the non-zero graded pieces of the weight filtration listed in the table below.
\begin{center}
\begin{tabular}{c|ccccc}
$X_d$ & $W_0$ &$\textup{Gr}_1^W$& $\textup{Gr}_2^W$&$\textup{Gr}_3^W$&$\textup{Gr}_4^W$ \\
\hline
$H^2_c$ & $1$ &     & $d-1$ &     &     \\
$H^3_c$ &     &     & $1$ &     &  \\
$H^4_c$ &     &     &     &     & $1$
\end{tabular}
\end{center}

\end{prop}

\begin{proof}
Applying the duality functor $\BD_{\BC^*}$ to the isomorphism \eqref{2d0001}, we obtain
\[
\BD_{\BC^*}Rf_*\BQ_X=\BD_{\BC^*}\BQ_{\BC^*}\oplus \BD_{\BC^*}(j_!\BQ_U[-1](-1)).
\]
By Lemma \ref{MHMdual}, the left side equals $Rf_!\BD_{X}\BQ_X=Rf_!\BQ_X[4](2)$. Similarly, the right side becomes $\BQ_{\BC^*}[2](1)\oplus Rj_*\BQ_U[3](2)$. Therefore, we have
\begin{equation} \label{2d0002}
Rf_!\BQ_{X}=\BQ_{\BC^*}[-2](-1)\oplus Rj_*\BQ_U[-1].
\end{equation}
By Theorem \ref{MHMdual}.(2) and Theorem \ref{MHM2}.(1), we have 
\begin{equation}\label{2d0003}
\begin{split}
\BH^*_c(\BC^*,Rj_*\BQ_U)=&\BH^*(\BC^*,j_!\BQ_U[2](1))^\vee\\
=&\left((\BQ^{\oplus d-1}[-1]\oplus\BQ[-1](-1))[2](1)\right)^\vee\\
=&\BQ^{\oplus d-1}[-1](-1)\oplus \BQ[-1].
\end{split}
\end{equation}
Therefore, the the weight filtration on $H^*_c(X_a,\BQ)$ follows from \eqref{2d0002} and \eqref{2d0003}
\[
H_c^*(X_a,\BQ)=\BQ[-2](-1)\oplus\BQ[-3](-2)\oplus \BQ^{\oplus d-1}[-2](-1)\oplus \BQ[-2].
\]
\end{proof}

\subsection{Geometry of $X_{a,b}$}
In this section, we study of the geometry 3-dimensional isolated cluster varieties of non-full rank, \emph{i.e.}
\[
X_{a,b}=\{(x,x',y,y',z)\in \BC^5\mid xx'=z^a+1,\, yy'=z^b+1,\,z\neq 0\}.
\]
We begin by analyzing the singularities of $X_{a,b}$. Let 
\[
\begin{cases}
g=xx'-z^a-1\\
h=yy'-z^b-1.
\end{cases}
\]
Then the Jacobian
\[
\frac{\partial(g,h)}{\partial(x,x',y,y',z)}=\begin{bmatrix}
x'&x&0 &0 &-az^{a-1} \\
0&0&y'&y & -bz^{b-1}
\end{bmatrix}.
\]
Since the Jacobian has only two rows, it is not of full rank if and only if the columns are all proportional to each other. Since $z,a,b\neq0$, the last column is a vector with nonzero entries. The other columns contain zero entries and are proportional to the last vector only if $x=x'=y=y'=0$. Therefore, the singularities occur exactly at 
\begin{equation} \label{singularity}
\{(x,x',y,y',z)\mid x=x'=y=y'=0,z^a=z^b=-1\}
\end{equation}

\begin{lem} \label{exp}
If $ord_2(a)=ord_2(b)$, then the cluster variety $X_{a,b}$ has isolated singularities, all of which are of type $A_1$, i.e. analytically locally isomorphic to the cone vertex over $\BP^1\times \BP^1$.
\end{lem}

\begin{proof}
From \eqref{singularity} and Lemma \ref{int}, the singularities if $X_{a,b}$ are isolated. 
Let $P_0=(0,0,0,0,z_0)$  be a singularity, and set $t=z-z_0$. Locally near $P_0$, the defining equations become
\[
\begin{cases}
xx'=(t+z_0)^a+1,\\
yy'=(t+z_0)^b+1.
\end{cases}
\]
Expanding and using  $z_0^a=z_0^b=-1$, we obtain
\[
\begin{cases}
xx'=at\cdot u(t),\\
yy'=bt\cdot v(t),
\end{cases}
\]
where $u(t)$ and $v(t)$ are polynomials in $t$ satisfying $u(0)=v(0)=1$. In particular, $u(t)$ and $v(t)$ are invertible in $\BC[[t]]$. Hence the singularity $P_0$ is locally analytically isomorphic to the origin in the hypersurface $bxx'=ayy'$ in $\BC^4$. Since the local equation is homogeneous, the origin is the cone vertex over the hyperplane $bxx'=ayy'$ in $\BP^3$, which is isomorphic to $\BP^1\times\BP^1$. Thus the singularity is of type $A_1$.  
\end{proof}

Combining Lemma \ref{int} and \ref{exp}, we have the following.

\begin{prop} \label{sing}
Let $X_{a,b}$ be the cluster variety, and  $ord_2$ be the 2-adic valuation of integers.
\begin{enumerate}
\item If $ord_2(a)\neq ord_2(b)$, then $X_{a,b}$ is smooth.
\item If $ord_2(a)= ord_2(b)$, then  $X_{a,b}$ has isolated singularities at
\[
Sing(X_{a,b})=\{(0,0,0,0,z)\mid z^{\gcd(a,b)}=-1\}.
\]
Moreover, all singularities are of  type $A_1$.
\end{enumerate}
\end{prop}

\subsection{Mixed Hodge structure on $H^*_c(X_{a,b})$}
By construction, the cluster variety $X_{a,b}$ fits in the fibered product 
\[
\begin{tikzcd}
X_{a,b}\arrow[r]\arrow[d] & X_a\arrow[d,"f_a"]\\
X_b\arrow[r,"f_b"] & \BC^*
\end{tikzcd}
\]
where $f_a$ and $f_b$ are defined in (\ref{f}). By applying the Theorem \ref{MHM2}.(3) to $f_a$ and $f_b$ over the base $\BC^*$ and Proposition \ref{3.3}, we obtain
\begin{equation}\label{ab}
\begin{split}
Rf_!\BQ_{X_{a,b}}=&Rf_{a!}\BQ_{X_a}\boxtimes_{\BC^*}Rf_{b!}\BQ_{X_b}=Rf_{a!}\BQ_{X_a}\otimes Rf_{b!}\BQ_{X_b}\\
=&\BQ_{\BC^*}[-4](-2)\oplus Rj_{a*}\BQ_{U_a}[-3](-1)\oplus Rj_{b*}\BQ_{U_b}[-3](-1)\\
&\oplus Rj_{a*}\BQ_{U_a}\otimes Rj_{b*}\BQ_{U_b}[-2].
\end{split}
\end{equation}

The following proposition computes the mixed Hodge structure of the last term of \eqref{ab}. 

\begin{prop} \label{fiberedprod}
The mixed Hodge structure on $\BH^*_c(\BC^*,Rj_{a*}\BQ_{U_a}\otimes Rj_{b*}\BQ_{U_b})$ is of mixed Hodge-Tate type. If $ord_2(a)\neq ord_2(b)$, the non-zero graded pieces of the weight filtration are listed below.
\begin{center}
\begin{tabular}{c|ccc}
 & $W_0$ &$\textup{Gr}_1^W$& $\textup{Gr}_2^W$ \\
\hline
$H^1_c$ & $1$ & & $a+b-1$ 
\end{tabular}
\end{center}
If $ord_2(a)=ord_2(b)$, the non-zero graded pieces of the weight filtration are listed below.
\begin{center}
\begin{tabular}{c|ccccc}
 & $W_0$ &$\textup{Gr}_1^W$& $\textup{Gr}_2^W$&$\textup{Gr}_3^W$&$\textup{Gr}_4^W$ \\
\hline
$H^1_c$ & $1$ &  & $a+b-1$ &  &  \\
$H^2_c$ &     &  &  &  & $(a,b)$ 
\end{tabular}
\end{center}
\end{prop}

\begin{proof}
Theorem \ref{MHM2}.(2) produces a distinguished triangle
\[
i_{a*}i^!_a\BQ_{\BC^*}\to\BQ_{\BC^*}\to Rj_{a*}\BQ_{U_a}\to .
\]
Since $i_a$ is a regular closed embedding of isolated points, we have $i^!_a\BQ_{\BC^*}=\BQ_{Z_a}[-2](-1)$, and hence 
\[
i_{a*}\BQ_{Z_a}[-2](-1)\to\BQ_{\BC^*}\to Rj_{a*}\BQ_{U_a}\to .
\]
Tensoring with $Rj_{b*}\BQ_{U_b}$, we obtain
\begin{equation} \label{3.700}
i_{a*}\BQ_{Z_a}\otimes Rj_{b*}\BQ_{U_b}[-2](-1)\to Rj_{b*}\BQ_{U_b}\to Rj_{a*}\BQ_{U_a}\otimes Rj_{b*}\BQ_{U_b} \to .
\end{equation}
By applying Theorem \ref{MHMdual}.(4), we have 
\[
i_{a*}\BQ_{Z_a}\otimes Rj_{b*}\BQ_{U_b}[-2](-1)=i_{a*}i^{*}_aRj_{b*}\BQ_{U_b}[-2](-1)
\]
To calculate the right side, we consider the distinguished triangle
\begin{equation}\label{3.701}
   \BQ_{Z_a}\to i^*_aRj_{b*}\BQ_{U_b}\to \BQ_{Z_a\cap Z_b}[-1](-1)\to,
\end{equation}
obtained by applying the exact functor $i^*_a$ to 
\[
\BQ_{\BC^*}\to Rj_{b*}\BQ_{U_b}\to \BQ_{Z_b}[-1](-1)\to.
\]

Since extension class of \eqref{3.701} vanishes by dimension reason, we conclude that 
   \[
   i^*_aRj_*\BQ_{U_b}=\BQ_{Z_a}\oplus\BQ_{Z_a\cap Z_b}[-1](-1),
   \]
and hence the distinguished triangle \eqref{3.700} becomes
\[
\BQ_{Z_a}[-2](-1)\oplus \BQ_{Z_a\cap Z_b}[-3](-2)\to Rj_{b*}\BQ_{U_b}\to Rj_{a*}\BQ_{U_a}\otimes Rj_{b*}\BQ_{U_b} \to.
\]
Now the long exact sequence of compactly supported cohomology, together with Corollary \ref{3.3}, implies 

\[
\begin{array}{ccccccc}
0 & \longrightarrow & 0 & \longrightarrow & 0 & \longrightarrow & \mathrm{H}^0_c(Rj_{a*}\mathbb{Q}_{U_a} \otimes Rj_{b*}\mathbb{Q}_{U_b}) \\[4pt]
  & \longrightarrow & 0 & \longrightarrow & \mathbb{Q} \oplus \mathbb{Q}^{\oplus b-1}(-1) & \longrightarrow & \mathrm{H}^1_c(Rj_{a*}\mathbb{Q}_{U_a} \otimes Rj_{b*}\mathbb{Q}_{U_b}) \\[4pt]
  & \longrightarrow & \mathbb{Q}^{\oplus a}(-1) & \longrightarrow & 0 & \longrightarrow & \mathrm{H}^2_c(Rj_{a*}\mathbb{Q}_{U_a} \otimes Rj_{b*}\mathbb{Q}_{U_b}) \\[4pt]
  & \longrightarrow & H^0(Z_a\cap Z_b)(-2) & \longrightarrow & 0 & &
\end{array}
\]
Therefore, we have
\[
\BH^k_c(\BC^*,Rj_{a*}\BQ_{U_a}\otimes Rj_{b*}\BQ_{U_b})=\begin{cases}
\BQ\oplus\BQ^{\oplus a+b-1}(-1)& k=1\\
\BQ^{(a,b)}(-2)& k=2\text{ and } ord_2(a)=ord_2(b)\\
0& \text{otherwise,}
\end{cases}
\]
where the second case uses Lemma \ref{int}. The dimensions of the graded pieces are as desired.
\end{proof}

The main theorem of this section is the complete description of $H_c^*(X_{a,b})$.
\begin{thm} \label{cpt}
The mixed Hodge structure on $H^*_c(X_{a,b},\BQ)$ is of mixed Hodge-Tate type. When $ord_2(a)\neq ord_2(b)$, the cluster variety $X_{a,b}$ is smooth, and the non-zero graded pieces of the weight filtration are listed below.
\begin{center}
\begin{tabular}{c|ccccccc}
$X_{a,b}$ & $W_0$ &$\textup{Gr}_1^W$& $\textup{Gr}_2^W$&$\textup{Gr}_3^W$&$\textup{Gr}_4^W$&$\textup{Gr}_5^W$&$\textup{Gr}_6^W$ \\
\hline
$H^3_c$ & $1$ &  & $a+b-1$ &  &  &  & \\
$H^4_c$ &     &  & $2$ &  & $a+b-2$ &  & \\
$H^5_c$ &     &  &   &  & $1$ &  & \\
$H^6_c$ &     &  &  &  &  &  & $1$\\
\end{tabular}
\end{center}

When $ord_2(a)= ord_2(b)$, the cluster variety $X_{a,b}$ is singular, and the non-zero graded pieces of the weight filtration are listed below.
\begin{center}
\begin{tabular}{c|ccccccc}
$X_{a,b}$ & $W_0$ &$\textup{Gr}_1^W$& $\textup{Gr}_2^W$&$\textup{Gr}_3^W$&$\textup{Gr}_4^W$&$\textup{Gr}_5^W$&$\textup{Gr}_6^W$ \\
\hline
$H^3_c$ & $1$ &  & $a+b-1$ &   &   &   &  \\
$H^4_c$ &  &  & $2$ &   & $a+b+(a,b)-2$ &   &  \\
$H^5_c$ &  &  &  &  & $1$ &   &  \\
$H^6_c$ &  &  &  &  &   &   & $1$\\
\end{tabular}
\end{center}
\end{thm}

\begin{proof}
   Follows from \eqref{ab}, Proposition \ref{3.3} and Proposition \ref{fiberedprod}.
\end{proof}

\subsection{Mixed Hodge structures on $IH^*(X_{a,b})$ and $H^*(X_{a,b})$}
In this section, we compute the mixed Hodge structures on $IH^*(X_{a,b})$ and $H^*(X_{a,b})$ for the cluster variety $X_{a,b}$. By Proposition \ref{sing}, the smoothness of $X_{a,b}$ is determined by whether $ord_2(a)\neq ord_2(b)$. We first treat the smooth case. 

\begin{prop} \label{sm}
Suppose $ord_2(a)\neq ord_2(b)$. Then $X_{a,b}$ is smooth and the mixed Hodge structure on $H^*(X)=IH^*(X)$ is of mixed Hodge-Tate type. The non-zero graded pieces of the weight filtration are listed below.
\begin{center}
\begin{tabular}{c|ccccccc}
$X_{a,b}$ & $W_0$ &$\textup{Gr}_1^W$& $\textup{Gr}_2^W$&$\textup{Gr}_3^W$&$\textup{Gr}_4^W$&$\textup{Gr}_5^W$&$\textup{Gr}_6^W$ \\
\hline
$H^0$ & $1$ &   &   &   &   &   &  \\
$H^1$ &   &   & $1$ &   &   &   &  \\
$H^2$ &   &   & $a+b-2$ &   & $2$ &   &  \\
$H^3$ &   &   &   &   & $a+b-1$ &   & $1$\\
\end{tabular}
\end{center}
\end{prop}

\begin{proof}
   Since $X_{a,b}$ is a smooth 3-dimensional algebraic variety, Theorem \ref{MHMdual}.(2) gives the duality isomorphism
   \[
   H^k(X_{a,b})\cong H^{6-k}_c(X_{a,b})^\vee(-3)
   \]
   of mixed Hodge structures. The result follows immediately from Theorem \ref{cpt}.
\end{proof}

We now address the singular case. When $ord_2(a)=ord_2(b)$, the cluster variety $X_{a,b}$ has $(a,b)$ singularities, all of which are of type $A_1$. We first establish the following lemma concerning the local behavior of an $A_1$ singularity.

\begin{lem} \label{A1}
Let $C=Cone(\BP^1\times \BP^1)$ be the affine cone over $\BP^1\times\BP^1$. Let $\pi:\widetilde{C}\to C$ be the resolution of $C$ by blowing up the vertex $p$. Then:
\begin{enumerate}
   \item The mixed Hodge structures on the compactly supported cohomology groups are
   \[
   H^k_c(\widetilde{C})=
   \begin{cases}
      \BQ(-1), & k=2,\\
      \BQ^{\oplus2}(-2), & k=4,\\
      \BQ(-3),& k=6,\\
      0,& \text{otherwise}
   \end{cases}
   \quad
   \textup{and}
   \quad 
   H^k_c({C})=
   \begin{cases}
      \BQ(-1),     & k=3,\\
      \BQ(-2), & k=4,\\
      \BQ(-3),& k=6,\\
      0,& \text{otherwise}.
   \end{cases}
   \]
   \item There is a non-split distinguished triangle of mixed Hodge modules
   \[
   \BQ_C\to R\pi_*\BQ_{\widetilde{C}}\to\BQ_p^{\oplus2}[-2](-1)\oplus \BQ_p[-4](-2)\to.
   \]
   \item  There is a decomposition of mixed Hodge modules   
   \[
   Rf_*\BQ_{\widetilde{C}}[3]=IC_{C}\oplus \BQ_p[1](-1)\oplus\BQ_p[-1](-2).
   \]
   \end{enumerate}
\end{lem}

\begin{proof}
   \begin{enumerate}
   \item Let $j:C^\circ=C\setminus \{p\}\to C$ be the open complement of the closed embedding $i:\{p\}\to C$. Since $C^\circ$ is the total space of the $\BC^*$-bundle of the bidegree $(-1,-1)$ line bundle over $\BP^1\times\BP^1$, the differential $d_2$ in the Leray spectral sequence for the fibration $C^\circ\to\BP^1\times\BP^1$ is non-trivial. Together with the Poincar\'e duality, we have 
   \[
   H^k(C^\circ,\BQ)=\begin{cases}
   \BQ, & k=0,\\
   \BQ(-1),& k=2,\\
   \BQ(-2),& k=3,\\
   \BQ(-3),& k=5,\\
   0, & \text{otherwise}
   \end{cases}
   \quad\textup{and}\quad
   H^k_c(C^\circ,\BQ)=\begin{cases}
   \BQ,     &k=0,\\
   \BQ(-1), & k=3,\\
   \BQ(-2),& k=4,\\
   \BQ(-3),& k=6,\\
   0, & \text{otherwise}.
   \end{cases}
   \]
   Now applying $R\Gamma_c$ to the distinguished triangle 
   \[
   j_!\BQ_{C^\circ}\to \BQ_C \to i_*\BQ_p\to,
   \]
   we obtain a long exact sequence, which implies $H_c^k(C^\circ,\BQ)\cong H_c^k(C,\BQ)$ for $k\ge 2$ and an exact sequence
   \[
   0\to H^0_c(C,\BQ) \to \BQ \to \BQ \to H^1_c(C,\BQ)\to 0.
   \]

Since $C$ is non-compact, we have $H^0_c(C,\BQ)=0$. So we have
\[
H^k_c({C})=
   \begin{cases}
      \BQ(-1)     & k=3\\
      \BQ(-2) & k=4\\
      \BQ(-3)& k=6\\
      0& \text{otherwise}.
   \end{cases}
\]
On the other hand, $\widetilde{C}$ is the total space of a line bundle over $\BP^1\times\BP^1$, so its mixed Hodge structure is isomorphic to that of the base $H^*(\BP^1\times\BP^1)$. The desired mixed Hodge structure on $H^*_c(\widetilde{C})$ then follows from Poincar\'e duality.
    \item Since $\pi:\widetilde{C}\to C$ has connected fibers, we have $R^0\pi_*\BQ_{\widetilde{C}}=\BQ_C$ and hence a distinguished triangle 
    \[
    \BQ_C\to R\pi_*\BQ_{\widetilde{C}}\to \tau_{\ge1}R\pi_*\BQ_{\widetilde{C}}\to.
    \]
    Because $\pi$ is an isomorphism away from $p$, the higher derived images are supported at $p$. By the proper base change theorem, the stalks at $p$ are 
    \[
    (R^k\pi_*\BQ_{\widetilde{C}})_p=H^k(\pi^{-1}(p),\BQ)=\begin{cases}
    \BQ  &  k=0\\
    \BQ^2(-1)  & k=2\\
    \BQ(-2) & k=4\\
    0 & \text{otherwise.}
    \end{cases}
    \]
    Hence $\tau_{\ge1}R\pi_*\BQ_{\widetilde{C}}=\BQ_p^2[-2](-1)\oplus \BQ_p[-4](-1)$. From part (1), we see that $H^k_c(C)$ is not subspaces of $H^k_c(\widetilde{C})$, so the distinguished triangle does not split.
    \item Since $\widetilde{C}$ is smooth of dimension 3, we have $IC_{\widetilde{C}}=\BQ_{\widetilde{C}}[3]$. By BBDG decomposition theorem, there is a direct sum decomposition
    \[
    Rf_*\BQ_{\widetilde{C}}[3]=IC_C\oplus F
    \]
    where $F$ is supported at $p$. By Deligne's formula we have 
    \[IC_C=\tau_{\le-1}(Rj_*\BQ_{C^\circ}[3]).\]
    Proceeding as in part (2), and comparing stalks at $p$, we obtain $F=\BQ_p[1](-1)\oplus\BQ_p[-1](-2)$. This completes the proof.
   \end{enumerate}
\end{proof}

In fact, after blowing up an isolated $A_1$ singularity on an arbitrary threefold $X$, the mixed Hodge structure on the compacted supported cohomology changes exactly as in the local model. More precisely, we have the following.

\begin{prop} \label{a1-global}
Let $X$ be a threefold and let $p$ be an isolated singularity of type $A_1$. Let $\pi:\widetilde{X}\to X$ be the blowup at $p$. Then the natural map $\pi^*_k:H^k_c(X)\to H^k_c(\widetilde{X})$ is an isomorphism for all $k\neq2,3,4$, and we have
\[
\begin{split}
H^2_c(\widetilde{X})&=H^2_c(X)\oplus\BQ(-1),\\
H^4_c(\widetilde{X})&=H^4_c(X)\oplus\BQ(-2),\\
H^3_c(X)&=H^3_c(\widetilde{X})\oplus\BQ(-1).
\end{split}
\]
In particular, if the mixed Hodge structure on $H^*_c(X)$ is of mixed Hodge-Tate type, so is that on $H^*_c(\widetilde{X})$.
\end{prop}

\begin{proof}
   The natural morphism $\BQ_X\xrightarrow{\pi^*} R\pi_*\BQ_{\widetilde{X}}$ is an isomorphism away from $p$ since $\pi$ is a resolution of singularity. Hence the cone of $\pi^*$ is supported at $p$. Let $f:U\to X$ be an analytic neighborhood of $p$ in $X$ such that $U$ is analytically isomorphic to the cone $C$ from Lemma \ref{A1}, and set $\widetilde{U}= \pi^{-1}(U)$. Although analytic neighborhoods do not carry mixed Hodge structures, their cohomology groups coincide with thos of the local mordal, and the bounded derived categories of constuctible sheaves are naturally equivalent. By Lemma \ref{A1}.(2), the cone of $\pi^*$ in $D^b_c(U)$ is exactly $\BQ_p^{\oplus2}[-2]\oplus \BQ_p[-4]$. Thus we have a morphism of distinguished triangles in $D^b_c(U)$
   \[
   \begin{tikzcd}
   \BQ_X\arrow[r] &R\pi_*\BQ_{\widetilde{X}}\arrow[r]& \BQ_p^{\oplus2}[-2]\oplus \BQ_p[-4]\arrow[r]& \ \\
   f_!\BQ_U\arrow[r]\arrow[u]& f_!R\pi_*\BQ_{\widetilde{U}}\arrow[r]\arrow[u]&\BQ_p^{\oplus2}[-2]\oplus \BQ_p[-4]\arrow[r]\arrow[u,equal] & . 
   \end{tikzcd}
   \]
Applying the derived functor $R\Gamma_c$, we have an exact sequence of vector spaces
\[
0\to H^4_c(X)\to H^4_c(\widetilde{X})\to \BQ\to0
\]
and a morphism of long exact sequences of vector spaces
   \[
   \begin{tikzcd}
   0\arrow[r]& H^2_c(X)\arrow[r,"\pi^*_2"] &H^2_c(\widetilde{X})\arrow[r,"r"]& \BQ^2\arrow[r,"\delta"]&H^3_c(X)\arrow[r,"\pi^*_3"] &H^3_c(\widetilde{X})\arrow[r]&0\\
   0\arrow[r]& H^2_c(U)\arrow[r]\arrow[u]& H^2_c(\widetilde{U})\arrow[r]\arrow[u] &\BQ^2\arrow[r]\arrow[u,equal]&H^3_c(U)\arrow[r]\arrow[u]& H^3_c(\widetilde{U})\arrow[u]\arrow[r]&0
   \end{tikzcd}
   \]
   Consequently, we have 
   \[
   \text{coker }\pi^*_2\oplus \ker\pi^*_3\cong\BQ^2.
   \]
   We now claim that both $r$ and $\delta$ are nonzero. By Lemma \ref{A1}.(1), we have $H^2_c(U)=0$, $H^2_c(\widetilde{U})=\BQ$, and the map $H^2_c(\widetilde{U})\to\BQ^2$ is not zero. By commutativity of the diagram, the map $r$ is also nonzero. To show $\delta$ is nonzero, it suffices, by chasing diagram, to prove that $H^3_c(U)\to H^3_c(X)$ is injective. Indeed, the distinguished triangle 
   \[
   f_!\BQ_U\to \BQ_X\to g_*\BQ_{X\setminus U}\to
   \]
   induces the long exact sequence
   \[
   \to H^2_c(U)\to H^2_c(X)\xrightarrow{\alpha}H^2_c(X\setminus U)\to H^3_c(U)\to H^3_c(X)\to.
   \]
   On one hand, Lemma \ref{A1}.(1) gives $H^2_c(U)=0$, so the map $\alpha$ is injective. On the other hand, since $U$ is a small neighborhood of $p$, we have $X\setminus U\cong X\setminus\{p\}$. The long exact sequence 
   \[
   \cdots\to H^1_c(p)\to H^2_c(X\setminus p)\to H^2_c(X)\to H^2_c(p)\to\cdots
   \]
   implies that $H^2_c(X)\cong H^2_c(X\setminus p)\cong H^2_c(X\setminus U)$. So $\alpha$ is an isomorphism by dimension reason and hence $H^3_c(U)\to H^3_c(X)$ is injective.  This proves the statement at the level of vector spaces. 
   
   To see that the result holds at the level of mixed Hodge structures, it suffices to note that the cone of $\BQ_X\to R\pi_*\BQ_{\widetilde{X}}$ is contributed by the $H^2$ and $H^4$ the exceptional divisor $\pi^{-1}(p)$, which is algebraically isomorphic to $\BP^1\times \BP^1$. Consequently,  $\textup{coker} \,\pi^*_2$ and $\ker \pi^*_3$ have weight 2, while $\textup{coker}\,\pi^*_4$ have weight 4.
\end{proof}

\begin{prop} \label{res}  
Suppose $ord_2(a)=ord_2(b)$, so that $X_{a,b}$ is singular. Let $\pi:\widetilde{X}_{a,b}\to X_{a,b}$ be the resolution of the singularities, fitting into the Cartesian diagram
   \begin{center}
   \begin{tikzcd} 
   \cup_{i=1}^{(a,b)} E_i \arrow[r] \arrow[d]& \widetilde{X}_{a,b}\arrow[d,"\pi"]\\
   Z_{(a,b)}\arrow[r]& X_{a,b},
   \end{tikzcd}
   \end{center}
  where $E_i$ are isomorphic to $\BP^1\times \BP^1$. The mixed Hodge structure on $H^*_c(\widetilde{X}_{a,b})$ is of mixed Hodge-Tate type, and the non-zero graded pieces of the weight filtration are listed below.
   \begin{center}
\begin{tabular}{c|ccccccc}
$\widetilde{X}_{a,b}$ & $W_0$ &$\textup{Gr}_1^W$& $\textup{Gr}_2^W$&$\textup{Gr}_3^W$&$\textup{Gr}_4^W$&$\textup{Gr}_5^W$&$\textup{Gr}_6^W$ \\
\hline
$H^2_c$ &   &   & $(a,b)$ &   &   &   &  \\
$H^3_c$ & $1$ &   & $a+b-(a,b)-1$ &   &   &   &  \\
$H^4_c$ &   &   & $2$ &   & $a+b+2(a,b)-2$ &   &  \\
$H^5_c$ &   &   &   &   & $1$ &   &  \\
$H^6_c$ &   &   &   &   &   &   & $1$\\
\end{tabular}
\end{center}

Dually, we obtain the cohomology version.
\begin{center}
\begin{tabular}{c|ccccccc}
$\widetilde{X}_{a,b}$ & $W_0$ &$\textup{Gr}_1^W$& $\textup{Gr}_2^W$&$\textup{Gr}_3^W$&$\textup{Gr}_4^W$&$\textup{Gr}_5^W$&$\textup{Gr}_6^W$ \\
\hline
$H^0$ & $1$ &   &   &   &   &   &  \\
$H^1$ &   &   & $1$ &   &   &   &  \\
$H^2$ &   &   & $a+b+2(a,b)-2$ &   & $2$ &   &  \\
$H^3$ &   &   &   &   & $ a+b-(a,b)-1$ &   & $1$\\
$H^4$ &   &   &   &   & $(a,b)$ &   &  
\end{tabular}
\end{center}
\end{prop}

\begin{proof}
   The mixed Hodge structure on $H^*_c(\widetilde{X}_{a,b})$ is obtained by applying Proposition \ref{a1-global} inductively to each singularity, starting from the Proposition \ref{cpt}. The cohomology version then follows from Poincar\'e duality.
\end{proof}

\begin{thm} \label{s}
  Suppose $ord_2(a)=ord_2(b)$, so that $X_{a,b}$ is singular. Then the mixed Hodge structure on $IH^*(X_{a,b})$ is of mixed Hodge-Tate type with the non-zero graded pieces listed below.
   \begin{center}
\begin{tabular}{c|ccccccc}
$X_{a,b}$ & $W_0$ &$\textup{Gr}_1^W$& $\textup{Gr}_2^W$&$\textup{Gr}_3^W$&$\textup{Gr}_4^W$&$\textup{Gr}_5^W$&$\textup{Gr}_6^W$ \\
\hline
$IH^0$ & $1$ &   &   &   &   &   &  \\
$IH^1$ &   &   & $1$ &   &   &   &  \\
$IH^2$ &   &   & $a+b+(a,b)-2$ &   & $2$ &   &  \\
$IH^3$ &   &   &   &   & $a+b-(a,b)-1$ &   & $1$\\
\end{tabular}
\end{center}
Similarly, the mixed Hodge structure on $H^*(X_{a,b})$ is of mixed Hodge-Tate type with non-zero  graded pieces listed below.
\begin{center}
\begin{tabular}{c|ccccccc}
$X_{a,b}$ & $W_0$ &$\textup{Gr}_1^W$& $\textup{Gr}_2^W$&$\textup{Gr}_3^W$&$\textup{Gr}_4^W$&$\textup{Gr}_5^W$&$\textup{Gr}_6^W$ \\
\hline
$H^0$ & $1$ &   &   &   &   &   &  \\
$H^1$ &   &   & $1$ &   &   &   &  \\
$H^2$ &   &   & $a+b-2$ &   & $2$ &   &  \\
$H^3$ &   &   &   &   & $ a+b-(a,b)-1$ &   & $1$\\
\end{tabular}
\end{center}
\end{thm}

\begin{proof}
Let $\pi:\widetilde{X}_{a,b}\to X_{a,b}$ be the resolution from Proposition \ref{res}. By the BBDG decomposition theorem, we have 
\[
R\pi_*\BQ_{\widetilde{X}_{a,b}}[3]=IC_{X_{a,b}}\oplus G
\]
where $G$ is supported on the singular locus $Z_{(a,b)}$. By Lemma \ref{A1}.(3), we have 
\[
G=\BQ_{Z_{(a,b)}}[1](-1)\oplus\BQ_{Z_{(a,b)}}[-1](-2). 
\]
Taking hypercohomology yields
\[
H^k(\widetilde{X}_{a,b})=IH^k(X_{a,b})\oplus 
\begin{cases}
\BQ^{\oplus(a,b)}(-1), &  k=2,\\
\BQ^{\oplus(a,b)}(-2), &  k=4, \\
0, & \text{otherwise.}
\end{cases}
\]
Combining with Proposition \ref{res}, we obtain the desired mixed Hodge structure on $IH^*(X)$.

The mixed Hodge structure on $H^*(X)$ follows from Proposition \ref{res} and \cite[Theorem 5.35]{PS}.
\end{proof}

\begin{rmk}
In \cite{ZZ}, we compute the mixed Hodge structures on $IH^*(X_{1,1})$ and $H^*(X_{1,1})$ via compactification and Deligne's spectral sequence. In the present paper, we instead computing them using the decomposition of the morphism $f:X\to \BC^*$, for two reasons. First, the morphism $f$ is naturally defined and is expected to admit natural generalization to higher dimensions. Second, the natural compactifications of general cluster varieties tend to have non-isolated singularities, making it difficult to develop a systematic approach to computing mixed Hodge structures via compactification. 
\end{rmk}

\section{P side}
\setcounter{subsection}{-1}
\subsection{Overview}
In \cite{Z}, the conjectural $P=W$ phenomenon for cluster varieties identifies the Hodge-theoretic weight filtrations of cluster varieties with the perverse filtrations associated with certain real torus fibrations. The cluster variety $X_{a,b}$ we studied in the previous section admit a natural real analytic map $h:X_{a,b}\to \BR^3$, where  
\begin{equation} \label{0001}
X_{a,b}=\{ (x,x^{\prime},y,y^{\prime},z)\in \mathbb{C}^5: xx^{\prime}=z^a+1,\,yy^{\prime}=z^b+1,\,z\neq0\}
\end{equation}
and $h:X_{a,b}\to \BR^3$ is defined by 
\begin{equation}
\begin{aligned}
h: X_{a,b} &\rightarrow \mathbb{R}^{3} \\
(x,x^{\prime},y,y^{\prime},z) &\mapsto (|x|^{2}-|x^{\prime}|^{2},
|y|^{2}-|y^{\prime}|^{2},\log|z|).
\end{aligned}
\label{2}
\end{equation}
In this section, we compute the perverse filtration associated with the map $h$.  We note that the BBDG decomposition theorem cannot be applied to $h$ since the $\BR^3$ is not a complex algebraic variety. Instead, we study the explicit geometry of the real analytic map $h$ and compute its perverse truncations directly. Our approach proceeds in following steps:
\begin{enumerate}
   \item Identify the singular fibers of $h$ and compute the monodromy groups of the smooth fibers around the critical values.  
   \item Describe $R^kh_*\BQ_X$ for $k=0,1,2,3$ explicitly, and show that the Leray spectral sequence of $h$ degenerate at $E_2$-page.
   \item Determine the perverse truncations of $Rh_*\BQ_X$ and compute the resulting perverse filtration associated with $h$.
\end{enumerate}
To simplify the notations, we simply write $X$ for $X_{a,b}$ in this section when no confusion arises.

\subsection{Geometry of the fibers}
In this section, we study the fibers of $h:X\to \BR^3$ by parametrizing the fibers by explicit angular coordinates.

Let $P=(u,v,w)\in\BR^3$. The fiber $h^{-1}(P)$ is the solutions to the system of equations
\begin{equation}
\begin{cases}
|x|^{2}-|x^{\prime}|^{2}=u,\\
|y|^{2}-|y^{\prime}|^{2}=v,\\
\log|z|=w,\\
xx^{\prime}=z^a+1,\\
yy^{\prime}=z^b+1.
\end{cases}
\label{3}
\end{equation}
Once $z=e^{w+i\psi}$ is fixed, the equations \eqref{3} decouple into two subsystems:
\begin{equation} \label{0987}
\begin{cases}
|x|^{2}-|x^{\prime}|^{2}=u,\\
xx^{\prime}=e^{a(w+i\psi)}+1,
\end{cases}
\quad
\textup{and}
\quad
\begin{cases}
|y|^{2}-|y^{\prime}|^{2}=v,\\
yy^{\prime}=z^{b(w+i\psi)}+1.
\end{cases}
\end{equation}
Since the two subsystems are analogous, we solve the first one for $x$ and $x'$ as an illustration.  From second equation, we have $|x|\cdot|x^{\prime}|=|e^{a(w+i\psi)}+1|$. Together with the first equation, this yields a system of quadratic equations of $|x|$ and $|x'|$. The solution is unique
\begin{equation} 
\begin{cases}
 \displaystyle |x| =\left(\frac{\sqrt{4|e^{a(w+\psi i)}+1|^{2}+u^{2}}+u}{2}\right)^{1/2},\\
 \displaystyle |x'| =\left(\frac{\sqrt{4|e^{a(w+\psi i)}+1|^{2}+u^{2}}-u}{2}\right)^{1/2},\\
  z=e^{w+\psi i}.\\
 \end{cases}
\end{equation}
The solutions for $|y|$ and $|y'|$ are obtained by replacing $a$ with $b$ and $u$ with $v$. To determine all solutions of $x$ and $x'$ it suffices to find their arguments. From \eqref{0987}, it follows that $|x|=0$ if and only if $u\le0,\,w=0,$ and $a\psi\in(2k+1)\BZ$. Similarly, $|x'|=0$ if and only if $u\ge0,\,w=0,$ and $a\psi\in(2k+1)\BZ$. Therefore, when $w\neq0$ or $u\neq0$, the variables $x$ and $x'$ are not simultaneously zero, so at least one of their arguments is well defined and can be chosen arbitrarily. The remaining variable is then determined by $xx'=z^a+1$. The same argument applies to $y$ and $y'$. More precisely, this leads to the following trivialization of $h$ over the regular values as a fiber bundle. 

\begin{prop}\label{5.2}
Let $h:X\to\BR^3$ be the real analytic map defined by \eqref{2}. Then $U=\BR^3\setminus \{(u,v,w)\mid uv=w=0\}$ is the open set of the regular values of $h$. Consequently, map $h$ defines a $T^3$-bundle over $U$. Moreover, $U$ is covered by the following four contractible open subsets
\[
\begin{split}
U_{++}=\{(u,v,w)\mid w\neq0\}\cup \{(u,v,w)\mid u>0,\, v>0\},\\
U_{+-}=\{(u,v,w)\mid w\neq0\}\cup \{(u,v,w)\mid u>0,\, v<0\},\\
U_{-+}=\{(u,v,w)\mid w\neq0\}\cup \{(u,v,w)\mid u<0,\, v>0\},\\
U_{--}=\{(u,v,w)\mid w\neq0\}\cup \{(u,v,w)\mid u<0,\, v<0\}.
\end{split}
\]
which trivializes $h$. The trivialization map over $U_{**}$, where $*=+$ or $-$, is given by
\[
\begin{tikzcd}
 h^{-1}U_{**}\arrow[rr]\arrow[rd]& & U_{**}\times T^3\arrow[ld]\\
     &U_{**}&
\end{tikzcd}
\]
which sends $(x,x',y,y',z)$ to $(|x|^2-|x'|^2,|y|^2-|y'|^2,\log|z|,\arg x^\Box,\arg y^\Box,\arg z)$ on $U_{**}$, where if $*=+$, $\Box=$ empty and if $*=-$, $\Box='$. 
\end{prop}

\begin{proof}
By symmetries between $x$ and $x'$, and between $y$ and $y'$, it suffices to prove the statement for $U_{++}$. For $P=(u,v,w)\in \mathbb{R}^3$ satisfies either $w\neq 0$ or $u,v>0$, we need to show that the map 
\begin{eqnarray*}
h^{-1}(P)&\to& T^3\\
(x,x',y,y',z)& \mapsto&(\arg x,\arg y,\arg z)
\end{eqnarray*}
is a diffeomorphism and varies smoothly on $P$. Let 
\[
T^3 = \{(\theta,\varphi,\psi)\mid \theta,\varphi,\psi \in S^1 = \mathbb{R}/2\pi\mathbb{Z}\},
\]
be a parametrization of the 3-torus $T^3$. Define 
\begin{equation}
\begin{aligned}
\Phi:T^3 &\rightarrow h^{-1}(u,v,w) \\
(\theta,\varphi,\psi) &\to (x,x^{\prime},y,y^{\prime},z),
\end{aligned}
\end{equation}
where 
\begin{equation} \label{++}
\begin{cases}
 \displaystyle x =e^{\theta i}\left(\frac{\sqrt{4|e^{a(w+\psi i)}+1|^{2}+u^{2}}+u}{2}\right)^{1/2},\\
 \displaystyle y =e^{\varphi i}\left(\frac{\sqrt{4|e^{b(w+\psi i)}+1|^{2}+v^{2}}+v}{2}\right)^{1/2},\\
 z=e^{w+\psi i},\\
 \displaystyle x'=\frac{z^a+1}{x}\\
 \displaystyle y'=\frac{z^b+1}{y},\\
\end{cases}
\end{equation}

We first verify that $\Phi$ is well-defined, \emph{i.e.} $x$ and $y$ are non-zero. In fact, if,  for instance, $x=0$, then both $u<0$ and $w=0$ must hold, which are excluded in the definition of $U_{++}$. Similarly, we have $y\neq0$. Hence $\Phi$ is well-defined. A direct check shows that \eqref{++} gives the inverse to the map $(x,x',y,y',z)\to (\arg x,\arg y,\arg z)$, so $\Phi$ is bijective. Moreover, $\Phi$ depends smoothly on the parameters $u,v,w$. Therefore, the $\Phi$ is a diffeomorphism which varies smoothly over $U_{++}$, thus trivializing the fibration $h$ as a $T^3$-bundle over $U_{++}$.
\end{proof}

The trivialization implies $h^{-1}(P)$ itself carries a natural fiber bundle structure, namely an $S^1\times S^1$ bundle over $S^1$, where the fiber coordinates are $\theta=\arg x$ and $\varphi=\arg y$, and the base coordinate is $\psi=\arg z$. However, when $u=w=0$ and $a\psi\in(2k+1)\BZ$ for some $k\in\BZ$, the equations \eqref{0987} force $x=x'=0$, so the fiber can no longer be parametrized as Proposition \ref{5.2}. In another word, some fibers of the $S^1\times S^1$ bundle are partially contracted.  More precisely, the singular fibers are described as follows.

\begin{prop} \label{5}
Let $h:X\to\BR^3$ be as in \eqref{2} and let $U$ be as in Proposition \ref{5.2}. Then the fibers of $h$ are singular outside $U$. 

\begin{enumerate}
    \item When $P$ lies on the $u$-axis excluding the origin, the fiber $h^{-1}(P)$ is topologically a 3-torus with $b$ 2-tori collapsed to $b$ singular circles. More precisely, the fiber is isomorphic to $S^1 \times (S^1 \times S^1 / \sim)$ where $(\varphi_1,\psi_1)\sim (\varphi_2,\psi_2)$ if and only if either 
    \begin{enumerate}
        \item $\varphi_1=\varphi_2$ and $\psi_1=\psi_2$ or
        \item $\psi_1=\psi_2=(2k+1)\pi/b$ for some $k\in\BZ$, with $\varphi_1,\varphi_2$ arbitrary.
    \end{enumerate} 
    \item When $P$ is on the $v$-axis excluding the origin, the result in (1) holds by replacing $b$ by $a$, and $\varphi$ by $\theta$. 
    \item The fiber of the origin $O$ can be described as follows. 
     \begin{enumerate}
        \item If $ord_2(a)\neq ord_2(b)$, then $f^{-1}(O)$ is a 3-torus with $a+b$ 2-tori collapsed to $a+b$ singular circles.
        \item If $ord_2(a)=ord_2(b)$, then $f^{-1}(O)$ is a 3-torus with $a+b-\gcd(a,b)$ 2-tori collapsed to $a+b-(a,b)$ singular circles, and $\gcd(a,b)$ 2-tori collapsed to $\gcd(a,b)$ singular points.      
     \end{enumerate}
    More precisely, the fiber is isomorphic to the quotient space \(S^1 \times S^1 \times S^1 / \approx\), where $(\theta_1,\varphi_1,\psi_1)\approx (\theta_2,\varphi_2,\psi_2)$ if and only if 
    \begin{enumerate}
       \item $\theta_1=\theta_2$, $\varphi_1=\varphi_2$ and $\psi_1=\psi_2$, or
       \item $\psi_1=\psi_2=(2k+1)\pi/b$ for some $k\in\BZ$, $\theta_1=\theta_2$, $\varphi_1,\varphi_2$ are arbitrary, or
       \item $\psi_1=\psi_2=(2k+1)\pi/a$ for some $k\in \BZ$, $\varphi_1=\varphi_2$, $\theta_1,\theta_2$ are arbitrary.
    \end{enumerate}
\end{enumerate}

\end{prop}

\begin{proof}
Since the statement (2) is analogous to (1), we prove only (1) and (3). Let $P=(u,0,0)\in V_u$ lie on the $u$-axis with the origin excluded. Without loss of generality, we assume $u>0$. Define 
\begin{equation} \label{Phiu}
\begin{aligned}
\Phi_u:T^3 &\rightarrow h^{-1}(u,0,0) \\
(\theta,\varphi,\psi) &\to (x,x^{\prime},y,y^{\prime},z),
\end{aligned}
\end{equation}
by 
\begin{equation}\label{28}
\begin{cases}
 \displaystyle x =e^{\theta i}\left(\frac{\sqrt{4|e^{a\psi i}+1|^{2}+u^{2}}+u}{2}\right)^{1/2},\\
  z=e^{\psi i},\\
 \displaystyle x'=\frac{e^{a\psi i}+1}{x},\\
  y =\begin{cases}
       e^{\varphi i}\sqrt{|e^{b\psi i}+1|}, & \text{ if }\psi\not\in (2k+1)\pi/b,\\
    0, & \text{ if }\psi\in (2k+1)\pi/b,
  \end{cases}\\
  y'=\begin{cases}
     \displaystyle \frac{e^{b\psi i}+1}{y}, &  \text{ if }\psi\notin (2k+1)\pi/b,\\
     0, & \text{ if }\psi\in (2k+1)\pi/b.\\
  \end{cases}
\end{cases}
\end{equation}
In fact, the definition of $\Phi_u$ is exactly the same as $\Phi$ in \eqref{++} for $y\neq0$, and for the value of $y=0$, it is defined as the limit when $\psi\to (2k+1)\pi/b$ for some $k\in\BZ$. So $\Phi_u$ is continuous. As in the smooth case, the map $\Phi_u$ is a well-defined surjection since the solutions to \eqref{3} can be expressed as \eqref{28}. 

Now suppose two distinct points \((\theta_1, \varphi_1, \psi_1)\) and \((\theta_2, \varphi_2, \psi_2)\) in \(T^3\) are mapped to the same point via $\Phi_u$. Comparing the \(z\)-coordinates gives \(\psi_1 = \psi_2\). Comparing the $x$-coordinates and using the assumption $u>0$, we obtain \(\theta_1 = \theta_2\). Comparing the \(y\)-coordinates shows that if $b\psi\not=(2k+1)\pi$, then $\varphi_1=\varphi_2\). If $b\psi=(2k+1)\pi$ for some $k\in\BZ$, $\varphi_1$ and $\varphi_2$ can be chosen arbitrarily. This means $(\theta_1, \varphi_1, \psi_1) \sim (\theta_2, \varphi_2, \psi_2)$ if and only if either
\begin{enumerate}
    \item $\theta_1=\theta_2$, $\varphi_1=\varphi_2$ and $\psi_1=\psi_2$, or
    \item $\theta_1=\theta_2$, $\psi_1=\psi_2=(2k+1)\pi/b$ for some $k\in\BZ$, with $\varphi_1,\varphi_2$  arbitrary.
\end{enumerate}
Therefore, $h^{-1}(u,0,0)$ is the quotient space of $T^3$ by the equivalence relation above. Since $\theta$-coordinate is irrelevant to the other two, the fiber can be expressed as $S^1 \times (S^1 \times S^1 / \sim)$, where the three components correspond to $\theta, \varphi, \psi$, and the equivalence relation $\sim$ is exactly the one described in (1).

We now turn to the fiber $h^{-1}(O)$ over the origin. Analogous to the previous case, we define 

\begin{equation} \label{PhiO}
\begin{aligned}
\Phi_O:T^3 &\rightarrow h^{-1}(0,0,0) \\
(\theta,\varphi,\psi) &\to (x,x^{\prime},y,y^{\prime},z),
\end{aligned}
\end{equation}
by 
\begin{equation}
\begin{cases}
 z=e^{\psi i},\\
 x =\begin{cases}
 e^{\theta i}\sqrt{|e^{a\psi i}+1|},&\text{ if }\psi\not\in (2k+1)\pi/a,\\
 0, & \text{ if }\psi\in (2k+1)\pi/a,\\
 \end{cases}\\
   x'=\begin{cases} 
   \displaystyle \frac{e^{a\psi i}+1}{x}, &\text{ if }\psi\not\in (2k+1)\pi/a,\\
   0, &\text{ if }\psi\in (2k+1)\pi/a,\\
 \end{cases}\\   
  y =\begin{cases}
       e^{\varphi i}\sqrt{|e^{b\psi i}+1|}, & \text{ if }\psi\not\in (2k+1)\pi/b,\\
    0, & \text{ if }\psi\in (2k+1)\pi/b,
  \end{cases}\\
  y'=\begin{cases}
     \displaystyle \frac{e^{b\psi i}+1}{y}, &  \text{ if }\psi\notin (2k+1)\pi/b,\\
     0, & \text{ if }\psi\in (2k+1)\pi/b.\\
  \end{cases}
\end{cases}
\end{equation}
The map $\Phi_{O}$ is continuous and surjective by the same reason as before. Two points $(\theta_1,\varphi_1,\psi_1)$ and $(\theta_2,\varphi_2,\psi_2)$ are mapped to the same point if and only if they are related by the equivalence relation $\approx$ in (3). Therefore, the fiber $h^{-1}(O)$ is homeomorphic to the quotient space $T^3/\approx$.
\end{proof}

\begin{cor} \label{5.1}
The map \(h:X\to\BR^3\) is a proper surjective morphism of real analytic varieties, with all fibers of dimension 3.
\end{cor}

\begin{proof}
Since both $X$ and $\BR^3$ are locally compact, the properness is equivalent to the compactness of all fibers of $h$. Then apply Proposition \ref{5.2} and \ref{5}. 
\end{proof}

\subsection{Cohomology of fibers and specialization map}
From the geometric description of the fibers of $h$ in Proposition \ref{5.2} and Proposition \ref{5}, we may view fibers as $S^1\times S^1$-fibrations $\pi:h^{-1}(P)\to S^1$. We will study the decomposition of the complexes $R\pi_*\BQ_{h^{-1}(P)}$ and compute the cohomology of the fibers. For this purpose, we need the following lemma.

\begin{lem} \label{S1}
   Let $j:\mathfrak{U}\to S^1$ be an open embedding such that the complement $S^1\setminus \mathfrak{U}$ is a finite set of $n$ points. 
   \begin{enumerate}
   \item The cohomology groups satisfy
   \[
   H^k(S^1,j_!\BQ_\mathfrak{U})=\begin{cases}
   \BQ^{n} & k=1,\\
   0 & \text{otherwise}.
   \end{cases}
   \]
   \item Suppose $F\in D^b_c(S^1)^{\le0}$ and $G\in D^b_c(S^1)$. Suppose there is a distinguished triangle 
   \[
   F\to G\to j_!\BQ_{\mathfrak{U}}[-1]\to.
   \]
   Then $G=F\oplus j_!\BQ_{\mathfrak{U}}[-1]$. 
   \end{enumerate}
\end{lem}

\begin{proof}
   Applying the hypercohomology to the distinguished triangle 
   \[
   j_!\BQ_{\mathfrak{U}}\to\BQ_{S^1}\to\BQ_{S^1\setminus\mathfrak{U}}\to, 
   \]
   we obtain the exact sequence
   \[
   \begin{split}
   0\to &H^0(S^1,j_!\BQ_{\mathfrak{U}})\to H^0(S^1,\BQ)\xrightarrow{r} H^0(S^1\setminus\mathfrak{U},\BQ)\\
   \to&  H^1(S^1,j_!\BQ_{\mathfrak{U}})\to H^1(S^1,\BQ)\to 0.
   \end{split}
   \]
   Since the restriction map $r$ is injective, we conclude that
   \[
      H^k(S^1,j_!\BQ_\mathfrak{U})=\begin{cases}
   \BQ^{n} & k=1\\
   0 & \text{otherwise}.
   \end{cases}
   \]
   The splitting of the distinguished triangle follows from 
   \[
   \Ext^1(j_!\BQ_\mathfrak{U}[-1],F) =\Hom(j_!\BQ_\mathfrak{U},F[2])=\Hom(\BQ_{\mathfrak{U}},j^*F[2])=H^2(\mathfrak{U},j^*F)=0,
   \]
   where the last equality is due to $\dim\mathfrak{U}=1$ and $j^*F\in D^b_c(S^1)^{\le0}$.
\end{proof}

\begin{prop}\label{singfiber}
   Let $h:X\to \BR^3$ be the map \eqref{2} and $U$ be as in Proposition \ref{5}. Let $V_u$ and $V_v$ be the $u$-axis and $v$-axis excluding the origin. The table below gives the Betti numbers of the fiber $f^{-1}(P)$ for $P$ in different strata of $\BR^3$.
    \begin{center}
\begin{tabular}{c||c|c|c|c}
$f^{-1}(P)$ &  $O$   &$V_u$& $V_v$&$U$ \\
\hline
$H^0$ & $1$ &  $1$ & $1$ & $1$ \\
$H^1$ & $1$ &  $2$ & $2$ & $3$ \\
$H^2$ & $a+b$  & $b+1$ & $a+1$ & $3$ \\
$H^3$ (if $X$ smooth)& $a+b$  & $b$ & $a$ & $1$\\
$H^3$ (if $X$ singular)& $a+b-(a,b)$& $b$& $a$& $1$
\end{tabular}
\end{center}
   
\end{prop}

\begin{proof}
When $P\in U$, the Betti numbers of $h^{-1}(P)\cong T^3$ is obvious. When $P\in V_u$, the fiber is $S^1\times(S^1\times S^1/\sim)$.  Since $S^1\times S^1/\sim$ homeomorphic to a necklace of $b$ spheres, its Betti numbers are $1,1,b$. Hence the Betti numbers of $f^{-1}(P)$ is $1,2,b+1,b$. The case $P\in V_v$ is similar. To study the fiber $F=h^{-1}(O)$, we consider the natural projection $\pi:h^{-1}(O)\to S^1$ such that the graph
\[
\begin{tikzcd}
   T^3\arrow[r,"\Phi_O"]\arrow[rd,swap,"\text{pr}_3"]& F=h^{-1}(O)\arrow[d,dotted,"\pi"]\\
    & S^1
\end{tikzcd}
\]
commutes. The map $\pi$ is well-defined because for any point on the fiber, the $\psi$-coordinate of its preimage under $\Phi_0$ is uniquely determined. General fibers of $\pi$ are $T^2$ with coordinates $\theta$ and $\varphi$. The equivalence relation $\approx$ implies that the $S^1$ directions with coordinate $\theta$ are collapsed to a point over $(2k+1)\pi/a$, and similarly $S^1$ directions with coordinate $\varphi$ are collapsed to a point over $(2k+1)\pi/b$. By tracking the classes $d\theta$ and $d\varphi$ varying among the fibers, we obtain
\begin{equation} \label{31}
\begin{split}
R^0\pi_*\BQ_F&=\BQ_{S^1}\\
R^1\pi_*\BQ_F&=j_{a!}\BQ_{\mathfrak{U}_a}\oplus j_{b!}\BQ_{\mathfrak{U}_b}\\
R^2\pi_*\BQ_F&=j_{a,b!}\BQ_{\mathfrak{U}_{a,b}}
\end{split}
\end{equation}
where $j_a:\mathfrak{U}_a=S^1\setminus\{(2k+1)\pi/a\mid k\in\BZ\}\hookrightarrow S^1$ and $j_{a,b}:\mathfrak{U}_{a,b}=\mathfrak{U}_{a}\cap \mathfrak{U}_{b}\hookrightarrow S^1$. Now apply Lemma \ref{S1}.(2) to the distinguished triangles 
\[
\BQ_{S^1}\to \tau_{\le1}R\pi_*\BQ_F\to j_{a!}\BQ_{\mathfrak{U}_a}\oplus j_{b!}\BQ_{\mathfrak{U}_b}[-1]
\]
and
\[
(\tau_{\le1} R\pi_*\BQ_F)[1]\to R\pi_*\BQ_F[1]\to j_{a,b!}\BQ_{\mathfrak{U}_{a,b}}[-1]\to,
\]
we have 
\[
R\pi_*\BQ_F=\BQ_{S^1}\oplus j_{a!}\BQ_{\mathfrak{U}_a}[-1]\oplus j_{b!}\BQ_{\mathfrak{U}_b}[-1]\oplus j_{a,b!}\BQ_{\mathfrak{U}_{a,b}}[-2].
\]
By Lemma \ref{int} and Proposition \ref{sing}, we have
\[
\# (S^1\setminus\mathfrak{U}_{a,b})=
\begin{cases}
 a+b & \text{if }X \text{ is smooth},\\
 a+b-(a,b) & \text{if }X \text{ is singular}.
\end{cases}
\]
Therefore, by Lemma \ref{S1}.(1), we have
\[
H^k(F,\BQ)=\begin{cases}
1 & k=0,1\\
a+b & k=2,3\\
0 & \text{otherwise,}
\end{cases} \quad \text{if }X \text{ is smooth}
\]
and
\[
H^k(F,\BQ)=\begin{cases}
1 & k=0,1\\
a+b & k=2\\
a+b-(a,b)& k=3\\
0 & \text{otherwise,}
\end{cases} \quad \text{if }X \text{ is singular}.
\]
\end{proof}

\begin{rmk}
By comparing the Betti numbers in Proposition \ref{sm}, Theorem \ref{s} and Proposition \ref{singfiber} we see that the restriction map $H^*(X)\to H^*(h^{-1}(O))$ is an isomorphism. So in the topological perspective, the torus fibration $h$ behaves similarly to the Hitchin-type elliptic fibrations in the sense of \cite{Z3}.  
\end{rmk}

Dually, the homology cycles of the fibers can be described as follows. The proof follows directly from Proposition \ref{singfiber}.

\begin{cor} \label{singfiberh}
   Let $h:X\to \BR^3$ be the map \eqref{2} and let $U$ be as in Proposition \ref{5}. Let $V_u$ and $V_v$ be the $u$-axis and $v$-axis excluding the origin. Then the homology cycles of the fibers can be described as follows.
   \begin{enumerate}
       \item When $P\in V_u^+$ on the positive $u$-axis, the homology group satisfies
       \[
       H_k(h^{-1}(P),\BQ)=
       \begin{cases}
       \langle\textup{pt}\rangle & k=0,\\
       \langle\Gamma_\theta,\Gamma_\psi\rangle & k=1,\\
       \langle S_1,\dots,S_b, \Gamma_\theta\times\Gamma_\psi\rangle & k=2,\\
       \langle S_1\times\Gamma_\theta,\cdots,S_b\times\Gamma_\theta\rangle & k=3,\\
       \end{cases}
       \]
       where $\Gamma_\theta$ is a circle with $\theta$ rotating from $0$ to $2\pi$ with $\varphi,\psi$ fixed, and $S_i$ is the locus defined by $(2i-1)\pi/b\le\psi\le (2i+1)\pi/b$ with $\theta$ fixed. When $P\in V_u^-$, replace $\theta$ by $\theta'$. 
       \item When $P\in V_v^+$ on the positive $v$-axis, the homology group satisfies
       \[
       H_k(h^{-1}(P),\BQ)=
       \begin{cases}
       \langle\textup{pt}\rangle & k=0,\\
       \langle\Gamma_\varphi,\Gamma_\psi\rangle & k=1,\\
       \langle T_1,\cdots,T_a, \Gamma_\varphi\times\Gamma_\psi\rangle & k=2,\\
       \langle T_1\times\Gamma_\varphi,\cdots,T_a\times\Gamma_\varphi\rangle & k=3,\\
       \end{cases}
       \]
       where $T_i$ is the locus defined by $(2i-1)\pi/a\le\psi\le (2i+1)\pi/a$ with $\varphi$ fixed. When $P\in V_v^-$, replace $\varphi$ by $\varphi'$.
       \item  The homology of the fiber $h^{-1}(O)$ satisfies      
       \[
       H_k(h^{-1}(O),\BQ)=
       \begin{cases}
       \langle\textup{pt}\rangle & k=0,\\
       \langle\Gamma_\psi\rangle & k=1,\\
       \langle S_1,\cdots,S_b, T_1,\cdots T_a\rangle & k=2,\\
       \langle R_1,\cdots,R_c\rangle & k=3,\\
       \end{cases}
       \]
       where $c=|Z_{a,b}|$ where $Z_{a,b}=\{\psi\mid e^{a\psi}=-1 \textup{ or } e^{b\psi}=-1\}$. If $0\le\psi_1<\cdots<\psi_c<2\pi$ are the elements of $Z_{a,b}$ in increasing order, then $R_i$ is the 3-cycle in $f^{-1}(O)$ defined by  $\psi_i\le\psi\le \psi_{i+1}$.
   \end{enumerate}
\end{cor}

The next goal is to understand the specialization map of the fibers of $h$. Let $P\in\BR^3$ be a point. For a sufficiently small Euclidean neighborhood $W$ of $P$ the restriction map
\[
H^*(h^{-1}(W))\to H^*(h^{-1}(P))
\]
is an isomorphism. The existence of such neighborhoods follows from the proper base change theorem and the constructibility of $Rh_*\BQ_X$. Then, for any point $Q\in W\setminus \{P\}$, there exists a natural map
\[
sp^*:H^*(h^{-1}(P),\BQ)\cong H^*(h^{-1}(W),\BQ)\to H^*(h^{-1}(Q),\BQ),
\]
called the specialization map. Equivalently, the dual map 
\[
sp_*:H_*(h^{-1}(Q),\BQ)\to H_* (h^{-1}(P),\BQ)
\]
describes how the homology classes of a nearby fiber $h^{-1}(Q)$ behave when pushed into the central fiber $h^{-1}(P)$. We remark that due to the constructibility, when $Q$ is sufficiently closed to $P$, the specialization map depends only on the strata containing $Q$ rather than the specific choice of $Q$. For our purpose, we will need to study the specialization map for points on the $u$-axis, $v$-axis and the origin. In fact, since the parametrization $\Phi_u$ and $\Phi_O$ in \eqref{Phiu} and \eqref{PhiO} are defined as the limit of the isomorphism $\Phi$ in \eqref{++} as the point $P$ approaches the $u$-axis and the origin, they describe how a general fiber degenerates to special fibers over $u$-axis and $O$. More precisely, we have the following description.

\begin{prop} \label{degenerate}
   Let $h:X\to \BR^3$ be the map \eqref{2} and $U$ be as in Proposition \ref{5}. Let $V_u$ and $V_v$ be the $u$-axis and $v$-axis excluding the origin.
   \begin{enumerate}
   \item Let $P$ be a point on $V_u^+$ and $Q$ be a point in $U$ near $P$ such that the specialization map $sp_*:H_*(h^{-1}(Q),\BQ)\to H_*(h^{-1}(P),\BQ)$  exists. Then $sp_*(\Gamma_\varphi)=0$, $sp_*(\Gamma_\theta)=\Gamma_\theta$, $sp_*(\Gamma_\psi)=\Gamma_\psi$, $sp_*(\Gamma_\theta\times \Gamma_\psi)=\Gamma_\theta \times \Gamma_\psi$, and $sp_*(\Gamma_\theta\times\Gamma_\psi)=\sum_{i=1}^b S_i$.  
   For $P\in V_u^-$, replace $\theta$ by $\theta'$.
      \item Let $P$ be a point on $V_v^+$ and $Q$ be a point in $U$ near $P$ such that the specialization map $sp_*:H_*(h^{-1}(Q),\BQ)\to H_*(h^{-1}(P),\BQ)$  exists. Then $sp_*(\Gamma_\theta)=0$, $sp_*(\Gamma_\varphi)=\Gamma_\varphi$, $sp_*(\Gamma_\psi)=\Gamma_\psi$, $sp_*(\Gamma_\varphi\times \Gamma_\psi)=\Gamma_\varphi \times \Gamma_\psi$, and $sp_*(\Gamma_\varphi\times\Gamma_\psi)=\sum_{i=1}^a T_i$.   
   For $P\in V_v^-$, replace $\varphi$ by $\varphi'$.
    \item  Let $O$ be the origin and $P$ be a point on $V_u^+$ near $O$ such that the specialization map $sp_*:H_*(h^{-1}(O),\BQ)\to H_*(h^{-1}(P),\BQ)$  exists. Then $sp_*(\Gamma_\varphi)=0$, $sp_*(\Gamma_\theta\times \Gamma_\psi)=\sum_{i=1}^a T_i$, and $sp_*(S_i\times\Gamma_\theta)=\sum R_j$, where the cycle $R_j$ appears in the sum if and only if it satisfies $(2i-1)\pi/b\le\psi\le (2i+1)\pi/b$. When $P\in V_u^-,V_v^+,V_v^-$, similar results hold.
     \end{enumerate}
\end{prop}

\begin{proof}
   Without loss of generalization, it suffices to prove the statement for the specializations from $Q$ to $P$ and $P$ to $O$ for $P\in V_u^+$. By the construction, the specializations $sp_*$ are exactly induced by the push-forward of the map $\Phi_u$ and $\Phi_O$. By Proposition \ref{5}, $\Phi_u$ is the map 
   \[
   \Phi_u=\text{id}\times q:S^1_\theta\times S^1_\varphi\times S^1_\psi\to S^1_\theta\times (S^1_\varphi\times S^1_\psi/\sim)
   \]
   where the subscripts indicates the corresponding angular variables and $q$ is the quotient map which contracts $b$ copies of $\Gamma_\varphi$ on the 2-torus. Consequently, we have $q_*(\Gamma_\varphi)=0$, $q_*(\Gamma_\psi)=\Gamma_\psi$ and $q_*(\Gamma_\varphi\times\Gamma_\psi)=\sum_{i=1}^bS_i$. The result then follows from the K\"unneth formula.

   Similarly, the specialization map $sp_*$ from $P\in V_u^+$ to $O$ is induced by the push-forward of the map 
   \[
   \begin{tikzcd}
   S^1_\theta\times (S^1_\varphi\times S^1_\psi)/\sim\arrow[rr,"sp"]\arrow[rd,"pr",swap]& & (S^1_\theta\times S^1_\varphi\times S^1_\psi)/\approx \arrow[ld,"pr"]\\
    &S^1_{\psi}.& 
    \end{tikzcd}
   \]
   where the equivalence relations $\sim$ and $\approx$ are defined in Proposition \ref{5}. Geometrically, if we view fibers $h^{-1}(P)$ and $h^{-1}(O)$ as $T^2$-fibrations over $S^1_\psi$, the specialization $sp$ exactly collapses all the $\Gamma_\theta$ cycles in the fibers over each $\psi=(2i+1)\pi/a$. So we obtain $sp_*(\Gamma_\theta)=0$, $sp_*(\Gamma_\psi)=\Gamma_\psi$, $sp_*(S_i)=S_i$, $sp_*(\Gamma_\theta\times\Gamma_\psi)=\sum_{i=1}^aT_i$. and $sp_*(\Gamma_\theta\times S_i)$ is exactly the sum of the $R_j$'s satisfying $(2k-1)\pi/b<\psi<(2k+1)\pi/b$. 
   \end{proof}

\subsection{Monodromy of $R^1h_*\BQ_X$}
In this section, we compute the monodromy action of the fundamental group $\pi_1(U)$ on $H^*(F,\BQ)$ where $F$ is any smooth fiber of $h$ and $U$ is the open set of regular values of $h$. We remark that, due to the ring structure of $H^*(T^3)$, the monodromy action on the total cohomology group is completely determined its action on $H^1$. 

\begin{lem} \label{pi1}
The fundamental group of $U$ is a free group with three generators.
\end{lem}

\begin{proof}
Via the natural projection $\BR^3\setminus \{O\} \to S^2$, the open set \(U\) deformation retracts to the unit sphere with four points removed, namely $(\pm1,0,0)$ and $(0,\pm1,0)$. Therefore, the fundamental group $\pi_1(U)$ is a free group on three generators.  
\end{proof}

For the convenience of presenting the monodromy as matrices, we fix the following set of generators. The directions of the loops and the corresponding axes follow the right-hand rule. 
\begin{enumerate}
\item The loop $\gamma_1$ goes around the \(u\)-axis at \(u=1\). 
\item The loop $\gamma_2$ goes around the \(v\)-axis at \(v=-1\).\
\item The loop $\gamma_3$ goes around the \(v\)-axis at \(v=1\). 
\end{enumerate}

\begin{thm} \label{5.3}
Let \(h: X \to \mathbb{R}^3\) be the map defined in \eqref{2}, and let
\[
U = \mathbb{R}^3 \setminus \{(u,v,w): uv = w = 0\}.
\]
Then \(h|_U: h^{-1}(U) \to U\) is a fiber bundle with fiber \(T^3\). Let $\theta=\arg x,\varphi=\arg y,\psi=\arg z$ be the angular variables in the trivialization over $U_{++}$ in Proposition \ref{5.2} and let $F$ be any fiber over $U_{++}$. With respect to the basis $\{d\theta,d\varphi,d\psi\}$ of $H^1(F,\BQ)$, the monodromy matrices for the loops $\gamma_1,\gamma_2,\gamma_3$ acting on $H^1(F,\BQ)$ are given by the matrices
\begin{equation*}
\begin{pmatrix}
1& 0& 0\\
0& 1& 0\\
0& b & 1
\end{pmatrix}
\quad\begin{pmatrix}
1& 0& 0\\
0& 1& 0\\
a& 0 & 1
\end{pmatrix}
\quad
\begin{pmatrix}
1& 0& 0\\
0& 1& 0\\
-a& 0& 1
\end{pmatrix}.
\end{equation*}
\end{thm}

\begin{proof}
Since the computations for the three loops are similar, we illustrate the argument for $\gamma_1$. Along this loop, we have $u>0$, so we use the trivializations $(\theta, \varphi, \psi)$ on $U_{++}$, and $(\theta, \varphi', \psi)$ on $U_{+-}$, where $\varphi'=\arg y'$. In de Rham cohomology, \(\{d\theta, d\varphi, d\psi\}\) and \(\{d\theta, d\varphi', d\psi\}\) are two bases for \(H^1(F)\). We first compute the transition data between these two bases. On the intersection $U_{++} \cap U_{+-}$, we always have $w \neq 0$. From \eqref{++} and $yy'=z^b+1$, we have
\begin{equation}
e^{(\varphi+ \varphi') i}=\frac{e^{b(w+\psi i)}+1}{|e^{b(w+\psi i)}+1|}. \label{9}
\end{equation}
Taking the logarithms and derivatives of both sides, we obtain
\begin{equation}
i \, d\varphi' = d\log\frac{e^{b(w+\psi i)}+1}{|e^{b(w+\psi i)}+1|} - i \, d\varphi. \label{10}
\end{equation}
In \(H^1(T^3)\), we may write \(d\varphi' = \lambda \, d\theta + \mu \, d\varphi + \nu \, d\psi\) with \(\lambda,\mu,\nu \in \mathbb{R}\). To determine these coefficients, we integrate over suitable 1-cycles in \(T^3\). Let $\Gamma_1$ be the cycle obtained by fixing \(\varphi, \varphi', \psi\) and letting \(\theta\) run from 0 to \(2\pi\). Integrating along $\Gamma_1$ implies that \(\lambda=0\). Let $\Gamma_2$ be the cycle defined by fixing \(\theta, \psi\) and letting \(\varphi\) run from 0 to \(2\pi\). Integrating both sides of \eqref{10} along $\Gamma_2$ implies \(\int_{\Gamma_2} d\varphi' = -\int_{\Gamma_2} d\varphi\). So \(\mu=-1\). Let $\Gamma_3$ be the cycle defined by fixing \(\theta, \varphi\) and letting \(\psi\) run from 0 to \(2\pi\). Integrating both sides of \eqref{10} along \(\Gamma_3\), we have
\begin{equation}
2\pi i \nu = \int_{\psi=0}^{\psi=2\pi} d\log\frac{e^{b(w+\psi i)}+1}{|e^{b(w+\psi i)}+1|}.
\end{equation}
In the complex $\psi$-plane, the curve \(\psi\mapsto e^{b(w+\psi i)}+1\) has winding number $b$ around the origin if $w>0$, and 0 if \(w<0\). Therefore, by the theorem of residue, we have 
\[
\int_{\psi=0}^{\psi=2\pi} d\log\frac{e^{b(w+\psi i)}+1}{|e^{b(w+\psi i)}+1|} = 
\begin{cases}
2b\pi i &w>0\\ 
0 & w<0.
\end{cases}
\]
So we have \(\nu=b\) if $w>0$, and $\nu=0$ if $w<0$.  In summary,
\begin{equation}
d\varphi' =
\begin{cases}
-d\varphi + b d\psi, & w>0,\\
-d\varphi, & w<0.
\end{cases}
\end{equation}
Now we compute the monodromy matrix of $\gamma_1$ on $H^1(F)$ by picking a moving point $(u,v,w)$ along loop $\gamma_1 = \overrightarrow{P_1P_2P_3P_4P_1}\) around the positive \(u\)-axis, where \(P_1=(1,0,\varepsilon), P_2=(1,-\epsilon,0), P_3=(1,0,-\varepsilon), P_4=(1,\epsilon,0)\). Among these points, only \(P_2\) is not in \(U_{++}\), and \(P_4\) is not in \(U_{+-}\). We have to use the trivialization $\theta,\varphi,\psi$ on $P_1, P_3,P_4$ and $\theta,\varphi',\psi$ on $P_1,P_2,P_3$. Starting at \(P_1\) with basis \(\{d\theta, d\varphi, d\psi\}\), we perform parallel transport along the flat sections induced by the trivialization, changing the basis before arriving $P_2$. Since \(w>0\), the initial basis equals \(\{d\theta, -d\varphi' + b\,d\psi, d\psi\}\). This basis is transported to \(P_3\). Since \(w<0\), we switch to \(\{d\theta, d\varphi + b\,d\psi, d\psi\}\),  passing through \(P_4\) and return to \(P_1\). Tracing through the changes of basis after traveling along \(\gamma_1\), we find that the monodromy matrix with respect to the basis \(\{d\theta, d\varphi, d\psi\}\) in \(H^1(F)\) is
\begin{equation}
\begin{pmatrix}
1& 0& 0\\
0& 1& 0\\
0& b& 1
\end{pmatrix}.
\end{equation}
The computations for the other two monodromy matrices are similar.
\end{proof}

\begin{cor}\label{mono2}
With respect to the basis $\{d\theta d\varphi,d\theta d\psi,d\varphi d\psi\}$, the matrices of the monodromy actions of $\gamma_1,\gamma_2,\gamma_3$ on $H^2(F,\BQ)$ are
\begin{equation*}
\begin{pmatrix}
1& 0& 0\\
b& 1& 0\\
0& 0 & 1
\end{pmatrix},
\quad\begin{pmatrix}
1& 0& 0\\
0& 1& 0\\
-a& 0 & 1
\end{pmatrix},
\quad
\begin{pmatrix}
1& 0& 0\\
0& 1& 0\\
a& 0& 1
\end{pmatrix}.
\end{equation*}
The monodromy action on $H^3(F,\BQ)$ is trivial.
\end{cor}

\begin{proof}
  Follows directly from $H^k(F,\BQ)=\wedge^kH^1(F,\BQ)$ and Theorem \ref{5.3}.
\end{proof}

\subsection{Computation of $R^kh_*\BQ_X$}
In this section, we compute the sheaves of $Rh_*\BQ_X$ for $h:X\to\BR^3$ and their cohomology groups. We will also show that the corresponding Leray spectral sequence 
\[
H^q(X,R^pf_*\BQ_X)\Rightarrow H^{p+q}(X,\BQ)
\]
degenerates on $E_2$-page. 
We fixed some notations first. Denote by $L$ the local system defined by the monodromy representation $\pi_1(U)\to \textup{GL}(H^1(F,\BQ))$ obtained in Section 4.3; equivalently, $L=R^1h_*\BQ_X|_U$. The next three propositions provide explicit descriptions of the cohomology sheaves $R^kh_*\BQ_X$.

\begin{prop} \label{rh1}
We have
\[
R^1h_*\BQ_X\cong j_*L.
\]
\end{prop}

\begin{proof}
Since $L=j^*R^1h_*\BQ_X$, the adjunction gives a morphism of sheaves 
\[
\eta:R^1h_*\BQ_X\to j_*L.
\]
To show $\eta$ is an isomorphism it suffices to verify that it induces isomorphisms on the stalks on $u$ or $v$ axis. Without loss of generality, we consider the case $P\in V_u^+$ and the origin $O$. By the proper base change theorem and the constructibility of $R^1h_*\BQ_X$, there exists a small analytic neighborhood $B$ of $P$ such that $H^1(h^{-1}(P),\BQ)=H^1(B,\BQ)$. The induced morphism on the stalk is
\[
\eta_P:H^1(h^{-1}(P),\BQ)=H^1(h^{-1}(B),\BQ)\to H^1(h^{-1}(B\cap U),\BQ)\to H^0(B\cap U,L)=(j_*L)_P
\]
where the second arrow is the canonical quotient map induced by the Leray spectral sequence for $h^{-1}(B\cap U)\to B\cap U$. Moreover, for any $Q\in B\cap U$, the restriction $H^0(B\cap U,L)\to L_Q=H^1(h^{-1}(Q),\BQ)$ fits into the following commutative diagram
\[
\begin{tikzcd}
(R^1h_*\BQ_X)_P\arrow[d,"\eta_P"]\arrow[r,equal]&H^1(h^{-1}(P),\BQ)\arrow[d,"\eta_P"]\arrow[rd,"sp^*"]& \\
(j_*L)_P\arrow[r,equal] &H^0(B\cap U,L)\arrow[r,"res"] & H^1(h^{-1}(Q),\BQ).
\end{tikzcd}
\]
By Proposition \ref{degenerate}, the specialization $sp_*:H_1(h^{-1}(Q),\BQ)\to H_1(h^{-1}(P),\BQ)$ is surjective with 1-dimensional kernel. So it's dual $sp^*$ is therefore injective with $\dim \textup{im}\,sp^*=2$. On the other hand, the restriction map $res$ maps isomorphically onto the monodromy invariant part of $H^1(h^{-1}(Q),\BQ)$, which is also of dimension 2 by Theorem \ref{5.3}. Hence $\eta_P$ is an injective linear map between two vector spaces both of dimension 2. We conclude that $\eta_P$ is an isomorphism. 

For the origin $O$, the isomorphism follows from the fact that $H^1(h^{-1}(O),\BQ)$ and $H^0(B\cap U,L)$ are both 1 dimensional and have the same image in $H^1(h^{-1}(O),\BQ)$.
\end{proof}

\begin{prop} \label{rh2}
We have the following decomposition
 \[
    R^2h_*\BQ_X\cong j_*\Lambda^2L\oplus \BQ^{\oplus b-1}_u\oplus \BQ_v^{\oplus a-1}, 
 \]
 where $\BQ_u$ and $\BQ_v$ denote the constant sheaves supported on the $u$-axis and $v$-axis, respectively.
\end{prop}

\begin{proof}
   As in the statement in the proof of Proposition \ref{rh1}, adjunction yields a morphism of sheaves
   \[
   \eta:R^2h_*\BQ_X\to j_*\Lambda^2L,
   \]
   whose kernel and cokernel are supported on the union of $u$ and $v$-axes.  Similarly, we have the following commutative diagram
   \begin{equation} \label{rh2d}
   \begin{tikzcd}
   (R^2h_*\BQ_X)_P\arrow[d,"\eta_P"]\arrow[r,equal]&H^2(h^{-1}(P),\BQ)\arrow[d,"\eta_P"]\arrow[rd,"sp^*"]& \\
   (j_*\Lambda^2 L)_P\arrow[r,equal] &H^0(B\cap U,\Lambda^2L)\arrow[r,"res"] & H^2(h^{-1}(Q),\BQ).
   \end{tikzcd}
   \end{equation}
   It suffices to analyze the stalk at points $P$ on $u$-axis and at the origin $O$. For $P\in V_u^+$, Proposition \ref{degenerate} implies $sp_*(\Gamma_\theta\times\Gamma_\psi)=\Gamma_\theta\times\Gamma_\psi$ and $sp_*(\Gamma_\theta\times\Gamma_\psi)=\sum_{i=1}^b S_i$. So $\textup{im}\,sp^*$ is 2-dimensional and 
   the kernel of $sp^*$ is spanned by the linear functionals $l:H_2(h^{-1}(P))\to \BQ$ satisfying $l(\Gamma_\theta\times\Gamma_\psi)=l\left(\sum_{i=1}^bS_i\right)=0$. On the other hand, the restriction map $res$ is an isomorphim onto the monodromy invariants of $H^2(h^{-1}(Q),\BQ)$, which is 2-dimensional by Corollary \ref{mono2}. So $\eta_P$ is surjective and $\ker \eta_P=\BQ^{b-1}$ spanned by linear functionals on spheres $S_1,\dots,S_b$ whose total sum is zero. Since the definition of $S_1,\cdots, S_b$ depends only on the variable $\psi$ on each fiber, this construction works for $P\in V_{u^-}$ and keeps constant when $P$ varies on $u$-axis. Therefore $\ker\eta|_{u^\pm}=\BQ_{u^\pm}^{b-1}$ and we obtain an exact sequence
   \begin{equation} \label{rh2ex}
   0\to \BQ_u^{b-1}\to (\ker \eta)|_{V_u}\to F\to 0
   \end{equation}
   where $F$ supported at the origin. By dimensional of the support we have $\Ext^1(F,\BQ_u^{b-1})=0$. So the sequence \eqref{rh2ex} splits and we have a morphism 
   \[
   \ker\eta\to(\ker \eta)|_{V_u}\to\BQ_u^{b-1}.
   \]
   Similarly, we have $\ker \eta \to\BQ_{v}^{a-1}$ and  
   \begin{equation} \label{rhm}
   \ker\eta\to\BQ_u^{b-1}\oplus \BQ_v^{a-1}.
   \end{equation}
   Applying the same argument of the diagram \eqref{rh2d} to $P=O$, the kernel of the specialization map $sp^*:H^2(h^{-1}(Q),\BQ)\to H^2(h^{-1}(O),\BQ)$ is spanned by the linear functional $l$ on the space spanned $S_1,\cdots,S_b,T_1,\dots,T_a$ such that the $l(\sum S_i)=l(\sum T_i)=0$. So the morphism \eqref{rhm} is an isomorphism at the origin. Therefore we obtain an isomorphism $\ker\eta=\BQ_u^{b+1}\oplus\BQ_v^{a+1}$, and hence an exact sequence
   \[
   0\to \BQ_u^{b-1}\oplus\BQ_v^{a-1}\to R^2h_*\BQ_X\to j_*\Lambda^2L\to0
   \]
   It remains to show the extension class $\Ext^1(j_*\Lambda^2L,\BQ_u^{b-1}\oplus\BQ_v^{a-1})$ vanishes. By symmetry it suffices to show that $\Ext^1(j_*\Lambda^2L,\BQ_u)=0$. Indeed, Corollary \ref{mono2} gives $i_u^*j_*\Lambda^2L=\BQ_u^2$, where $i_u:V_u\to \BR^3$ is the natural closed immersion. Therefore 
   \[
   \Ext^1_{D(\BR^3)}(j_*\Lambda^2L,\BQ_u)=0=\Ext^1_{D(V_u)}(i_u^*j_*\Lambda^2L,\BQ_u)=\Ext^1_{D(V_u)}(\BQ_u^2,\BQ_u)=0.
   \]
\end{proof}

\begin{prop} \label{rh3}
We have a decomposition of sheaves 
\[
R^3h_*\BQ_X\cong\BQ_{\BR^3}\oplus H,
\]
where $H=\BQ_u^{b-1}\oplus\BQ_v^{a-1}\oplus\BQ_O$ when $ord_2(a)\not=ord_2(b)$, and
 $H$ fits into the exact sequence 
\begin{equation} \label{1999}
0\to H\to \BQ_u^{b-1}\oplus\BQ_v^{a-1}\to \BQ_O^{(a,b)-1}\to0.
\end{equation}
when $ord_2(a)=ord_2(b)$.
\end{prop}

\begin{proof}
   As in the proof of Proposition \ref{rh1} and \ref{rh2}, there is a natural morphism 
   \[
   \eta:R^3h_*\BQ_X\to j_*\Lambda^3L=\BQ_{\BR^3}.
   \]
   Since the fundamental classes of fibers are always nonzero, the morphism $\eta$ is surjective, and hence we have the exact sequence
    \begin{equation}\label{rh3ex}
    0\to \ker\eta\to R^3f_*\BQ_X\to \BQ_{\BR^3}\to 0
    \end{equation}
   Let $P$ be a point on $u^+$-axis. By Proposition \ref{degenerate}, the kernel $\ker \eta_P$ is the vector space of the linear functionals on the vector space $\langle \Gamma_\theta\times S_1,\dots,\Gamma_\theta\times S_b\rangle$ such that $l(\sum\Gamma_\theta\times S_i)=0$. By the same argument as in Proposition \ref{rh2}, these functionals extend across $O$ to the $u^-$-axis and produce a morphism $\ker \eta\to\BQ_u^{b-1}$. Together with the analogous morphism $\ker\eta\to\BQ_v^{a-1}$, we have 
             \[\rho:\ker \eta\to \BQ_u^{b-1}\oplus \BQ_v^{a-1},\]
    which is isomorphic away from the origin $O$. At the origin, Proposition \ref{singfiberh} shows that the kernel $\ker\eta_O$ consists of all linear functionals $l$ on the $\BQ$-vector space spanned by $R_1,\dots,R_c$ such that $l(\sum R_i)=0$. By Corollary \ref{comb}, the behavior of $\eta_O$ depends on whether $ord_2(a)=ord_2(b)$. 
    
    \begin{enumerate}
    \item $ord_2(a)\neq ord_2(b)$. By Corollary \ref{comb}, the kernel of $\rho$ is $\BQ_O$, and $\rho$ is surjective by dimension reason. So 
    \[
    0\to \BQ_O\to\ker\eta\to \BQ_u^{b-1}\oplus\BQ_v^{a-1}\to 0
    \]
    is exact. The extension class 
    \begin{equation}\label{rh3pr}
    \begin{split}
    \Ext^1_{D(\BR^3)}(\BQ_u^{b-1}\oplus \BQ_v^{a-1},\BQ_O)&=\Hom_{D(\BR^3)}(\BQ_u^{b-1}\oplus \BQ_v^{a-1},i_*\BQ_O[1])\\
    =\Hom_{D(O)}(i^*\BQ_u^{b-1}&\oplus i^*\BQ_v^{a-1},\BQ_O[1])=\Hom_{D(O)}(\BQ_O^{a+b-2},\BQ_O[1])=0.
    \end{split}
    \end{equation}
    So $\ker\eta=\BQ_u^{b-1}\oplus \BQ_v^{a-1}\oplus\BQ_O$. The extension class of the exact sequence \eqref{rh3ex}    is 
    \[
    \Ext^1_{D(\BR^3)}(\BQ_O\oplus\BQ_u^{b-1}\oplus\BQ_v^{a-1},\BQ_{\BR^3})=0
    \]
    by the same reason as \eqref{rh3pr}.  Therefore in this case
    \[
    R^3h_*\BQ_X=\BQ_{\BR^3}\oplus \BQ_u^{b-1}\oplus \BQ_v^{a-1}\oplus \BQ_O. 
    \]
    \item $ord_2(a)=ord_2(b)$. By Corollary \ref{comb}, $\rho$ is injective. A dimensional count shows that the cokernel is of dimension $(a,b)-1$. So we obtain the exact sequence 
    \[
    0\to \ker\eta\to \BQ_{u}^{b-1}\oplus \BQ_v^{a-1}\to\BQ_O^{(a,b)-1}\to 0.
    \]
    Taking the cohomology, we have $H^1(\BR^3,\ker\eta)=0$, so the extension class of \eqref{rh3ex} is
    \[
    \Ext^1_{\BR^3}(\BQ_{\BR^3},\ker \eta)=H^1(\BR^3,\ker\eta)=0.
    \]
    Thus the sequence splits, and the statement follows with $H=\ker \eta$.  
    \end{enumerate}
\end{proof}

The result below concerns the cohomology groups of the direct summands of $R^kf_*\BQ_X$ for $k=1,2,3$.
 
\begin{prop} \label{coh}
Let $h:X\to \BR^3$, $j:U\to\BR^3$, $L=R^1f_*\BQ_X|_U$, and $H$ be as above. The cohomology groups of the sheaves $j_*L$, $j_*\Lambda^2L$ and $H$ are the following.
 We have 
\[
H^k(\BR^3,j_*L)=\begin{cases}
\BQ, & k=0,\\
0, & \textup{otherwise,}
\end{cases}
\quad
H^k(\BR^3,j_*\Lambda^2 L)=\begin{cases}
\BQ^2, & k=0,\\
0,& \textup{otherwise,}
\end{cases}
\]
\[
H^k(\BR^3,H)=\begin{cases}
\BQ^{a+b-(a,b)-1}, &k=0\textup{ and } ord_2(a)=ord_2(b),\\
\BQ^{a+b-1},& k=0\textup{ and } ord_2(a)\neq ord_2(b),\\
0,& \textup{otherwise.}
\end{cases}
\]
\end{prop}

\begin{proof}
The cohomology of $H$ is already computed in the proof of Proposition \ref{rh3}. The computations for $j_{*}L$ and $j_{*}\Lambda^2L$ are essentially the same, so we illustrate the procedures of $j_*L$ only. Let $V\xhookrightarrow{i_V} V\cup U\xhookleftarrow{j_V}U$ and $O\xhookrightarrow{i_O} \BR^3\xhookleftarrow{j_O}U\cup V$ be the open and closed embeddings of attaching strata. 
Consider the distinguished triangle
\begin{equation}\label{coh1000}
j_{V*}L\to Rj_{V*} L\to R^{\ge1}j_{V*} L\to.
\end{equation}
Local monodromy matrices in Theorem \ref{5.3} implies $R^1j_{V*}L=\BQ_u^2\oplus\BQ_v^2$ and $R^kj_{V*}L=0$ for $k\ge2$. Thus \eqref{coh1000} becomes
\begin{equation} \label{2000}
j_{V*}L\to Rj_{V*}L\to (\BQ_u^2\oplus\BQ_v^2)[-1]\to.
\end{equation}
Similarly, since $j_*L=j_{O*}j_{V*}L$, we have the distinguished triangle
\begin{equation} \label{2001}
 j_*L\to Rj_{O*}(j_{V*}L)\to R^{\ge1}j_{O*}(j_{V*}L)\to.
\end{equation}
The complex $R^{\ge1}j_{O*}(j_{V*}L)$ is supported at $O$, so it decomposes as
\[
R^{\ge1}j_{O*}(j_{V*}L)\cong\bigoplus_{k\ge1} R^kj_{O*}(j_{V*}L)[-k]
\]
Therefore,
\begin{equation} \label{2002}
H^k(\BR^3,R^{\ge1}j_{O*}(j_{V*}L))=(R^{k}j_{O*}(j_{V*}L))_P=\varinjlim_{B\ni O} H^k(B\setminus O,j_{V*}L),
\end{equation}
where $B$ runs through a filtered system of Euclidean balls centered at $O$. On the other hand, applying the long exact sequence induced by the restriction of \eqref{2000} to $B\setminus O$, we have
\[
\begin{split}
H^k(B\setminus O,j_{V*}L)=&\ker \{H^1(B\setminus O,L)\to H^0(B\cap V_a)\oplus H^0(B\cap V_b)\}\\
=&\ker \{H^1(\BR^3\setminus O,L)\to H^0(V_a)\oplus H^0(V_b)\}\\
=&H^k(\BR^3\setminus O,j_{V*}L).
\end{split}
\]
Therefore, the colimit in \eqref{2002} is actually isomorphisms and hence the natural map
\[
H^k(\BR^3\setminus O,j_{V*}L)\to H^k(\BR^3,R^{\ge1}j_{O*}j_{V*}L)
\]
is an isomorphism for $k\ge1$. Now taking hypercohomology of \eqref{2001} yields 
$H^0(\BR^3,j_*L)=H^0(U,L)=\BQ$ and $H^k(\BR^3,j_{V*}L)=0$ for $k\ge 1$, as desired. The computation for $j_*\Lambda^2L$ follows by the same argument, using the corresponding local monodromy data.
\end{proof}

\begin{thm} \label{leray split}
The Leray spectral sequence 
\[
E_2^{p,q}=H^p(\BR^3,R^qh_*\BQ_X)\Longrightarrow H^{p+q}(X,\BQ)
\]
induced by the map $h:X\to\BR^3$ is concentrated on the column $p=0$ and hence degenerates at the $E_2$-page.
\end{thm}

\begin{proof}
   By Proposition \ref{rh1}, Proposition \ref{rh2}, Proposition \ref{rh3} and Proposition \ref{coh}, we have
   \[
   \dim H^p(\BR^3,R^qh_*\BQ_X)=0,~~ \forall p>0.
   \]
   Therefore, the nonzero terms of $E_2^{p,q}$ are concentrated on a single column, and hence the Leray spectral sequence degenerates at $E_2$.
\end{proof}

\begin{rmk}
   We don't know whether the decomposition 
   \[
   Rh_*\BQ_X=\bigoplus_{k=0}^3 R^kh_*\BQ_X[-k]
   \]
   holds in the derived category. Nevertheless, as we shall see, the explicit descriptions of $R^kf_*\BQ_X$ will suffice for computing the perverse filtration; their direct summands will be suitably reorganized to produce the perverse cohomology sheaves.
\end{rmk}

\subsection{Perverse filtration of $h$}
We work with upper middle perversity; see Section 2.3. The following lemma is a reformulation of Definition \ref{perv} for the complexes in $D^b_c(\BR^3)$ that are constructible with respect to the stratification induced by $h$. 

\begin{lem} \label{perv3}
   Let $\BR^3=U\sqcup V\sqcup O$ be a stratification where $U$ is the complement of $u$-axis and $v$-axis, $V$ is the union of the two axes with the origin $O$ removed, and $O$ is the origin. Let $F\in D^b_c(\BR^3)$ be a complex such that the restriction of $F$ to each of the stratum has locally constant cohomology sheaves. Then $F$ is a perverse sheaf if and only if the following conditions hold:
   \begin{enumerate}
   \item $j_U^*F\in D^{=-1}_c(U)$, 
   \item $i_V^*F\in D^{\le0}_c(V)$,  
   \item $i_V^!F\in D^{\ge0}_c(V)$,
   \item $i_O^*F\in D^{\le 0}_c(O)$,
   \item $i_O^!F\in D^{\ge 0}_c(O)$.
   \end{enumerate}
\end{lem}

\begin{prop} \label{perv4}
The complexes $H$, $j_*L[1]$ and $j_*\Lambda^2 L[1]$ are perverse sheaves.
\end{prop}

\begin{proof}
We verify the conditions in Lemma \ref{perv3} for each complex. Since $H$ is supported on $V\cup O$, conditions (1)-(4) are trivial. To check condition (5), consider the following distinguished triangle
\[
i_{O*}i_O^!H\to H\to Ri_{V*}i_V^*H\to
\]
induced by the open and closed embedding $V\xrightarrow{i_V} V\cup O\xleftarrow{i_O}O$. From this, we conclude that $i_{O*}i_O^!H\in D^{\ge0}_c({V\cup O})$, and hence $i_O^!H\in D^{\ge 0}_c(O)$.

The proof for $j_*L[1]$ and $j_*\Lambda^2L[1]$ are similar, so we just prove that $j_*L[1]$ is a perverse sheaf as an illustration.  Condition (1), (2) and (4) follow directly from the definition. To check condition (3), we note that the stratification
\[
V\xhookrightarrow{i_V} \BR^3\setminus O\xhookleftarrow{j_V}U
\]
which yields a distinguished triangle
\[
i_{V*}i^!_Vj_*L\to j_{V*}L\to Rj_{V*} L\to.
\]
Comparing with \eqref{2000}, we have an isomorphism 
\[
i^!_Vj_*L\cong(\BQ_u^2\oplus\BQ_v^2)[-2],
\]
which lies in the correct perverse degree $i^!_Vj_*L\in D_c^{\ge0}$ and hence (3) holds. Similarly, comparing the distinguished triangle
\[
i_{O*}i^!_Oj_*L\to j_*L\to Rj_{O*}j^*_Oj_*L\to,
\]
with \eqref{2001}, we have 
\[
i^!_Oj_*L[1]=\tau_{\ge1}Rj_{O*}(j_{V*}L)\cong\bigoplus_{k\ge1} R^kj_{O*}(j_{V*}L)[-k].
\]
In particular, we have the complex $i^!_Oj_*L[1]$ is concentrated in $[0,\infty)$, and hence the condition $(5)$ holds.
\end{proof}

Now we are ready to formulate the perverse truncation of $Rh_*\BQ_X$.

\begin{thm} \label{pervtr}
The perverse truncation of the complex $Rh_*\BQ_X$ is 
\[
^\Fp\tau_{\le k}Rh_*\BQ_X=\begin{cases}
0 & k\le0,\\
\BQ_{\BR^3} & k=1,\\
Cone(\alpha)[-1] &k=2,\\
Cone(\beta)[-1] & k=3,\\
Rh_*\BQ_X & k\ge4,
\end{cases}
\]
where $\alpha$ is the composition
\begin{equation} \label{alpha}
\tau_{\le2}Rh_*\BQ_X\to R^2h_*\BQ_X=(j_*\Lambda^2L\oplus \BQ_u^{b-1}\oplus \BQ_v^{a-1})[-2]\to j_*\Lambda^2L[-2].
\end{equation}
and $\beta$ is the composition
\begin{equation}\label{beta}
Rh_*\BQ_X\to R^3h_*\BQ_X[-3]\to \BQ_{\BR^3}[-3].
\end{equation}
Furthermore, the perverse cohomology sheaves are
\[
^\Fp\CH^k(Rh_*\BQ_X)=\begin{cases}
\BQ_{\BR^3}[1] & k=1,\\
j_*L[1]\oplus\BQ_u^{b-1}\oplus\BQ_v^{a-1} &k=2,\\
j_*\Lambda^2L[1]\oplus H& k=3,\\
\BQ_{\BR^3}[1] & k=4, \\
0 & \textup{otherwise}.
\end{cases}
\]
\end{thm}

\begin{proof}
By Proposition \ref{byhand}, it suffices to construct the the following diagram 
\[
\begin{tikzcd}
0\arrow[r]& \BQ_{\BR^3}\arrow[r,"\iota_0"]\arrow[d] & Cone(\alpha)[-1]\arrow[r,"\iota_1"]\arrow[d] & Cone(\beta)[-1]\arrow[r,"\iota_2"]\arrow[d]  & Rh_*\BQ_X\arrow[d]\\
& \CP_{1}[-1]\arrow[lu,"+"] & \CP_{2}[-2]\arrow[lu,"+"] & \CP_{3}[-3]\arrow[lu,"+"] & \CP_{4}[-4]\arrow[lu,"+"]  
\end{tikzcd}
\]
such that all the triangles are distinguished triangles and $\CP_k$ are perverse sheaves for all $k$. It is obvious that $\CP_1=\BQ_{\BR^3}[1]$ is a perverse sheaf by Lemma \ref{perv3}. To find $\CP_2$, we construct $\iota_0$ first. Apply the octahedral axiom to \eqref{alpha}, we have
\begin{equation} \label{P21}
\begin{tikzcd}
   & & & (\BQ_u^{b-1}\oplus \BQ_v^{a-1})[-1]\\
   \tau_{\le2}Rh_*\BQ_X\arrow[rd]\arrow[rr,"\alpha"]& & j_*\Lambda^2L[-2]\arrow[r]\arrow[ru] &Cone(\alpha)\arrow[u,dotted] \\
    & (j_*\Lambda^2L\oplus \BQ_u^{b-1}\oplus \BQ_v^{a-1})[-2]\arrow[ru]\arrow[rrd]& & \\
          & & & \tau_{\le1}Rh_*\BQ_X[1]\arrow[uu,dotted].
\end{tikzcd}
\end{equation}
We define $\iota_0$ as the composition
\begin{equation} \label{iota0}
\BQ_{\BR^3}=R^0h_*\BQ_X\to \tau_{\le1}Rh_*\BQ_X\to Cone(\alpha)[-1].
\end{equation}
where $\tau_{\le1}Rh_*\BQ_X[-1]\to Cone(\alpha)$ is the dotted arrow  obtained in \eqref{P21}. Apply the octahedral axiom to \eqref{iota0}, we have
\[
\begin{tikzcd}
   & & & (\BQ_u^{b-1}\oplus \BQ_v^{a-1})[-2]\\
   \BQ_{\BR^3}\arrow[rd]\arrow[rr,"\iota_0"]& & Cone(\alpha)[-1]\arrow[r]\arrow[ru] &\CP_2[-2]\arrow[u,dotted] \\
    & \tau_{\le1}Rh_*\BQ_X\arrow[ru]\arrow[rrd]& & \\
          & & & j_*L[-1].\arrow[uu,dotted]
\end{tikzcd}
\]
The extension class of the (shifted) dotted distinguished triangle 
\[
j_*L[1]\to \CP_2\to \BQ_u^{b-1}\oplus \BQ_v^{a-1}\to
\]
is 
\[
\begin{split}
\Ext^1(\BQ_u^{b-1}\oplus\BQ_v^{a-1},j_*L[1])=\Hom(\BQ_u^{b-1}\oplus\BQ_v^{a-1},j_*L[2])\\
=\Hom(j^*(\BQ_u^{b-1}\oplus\BQ_v^{a-1}),L[2])=\Hom(0,L[2])=0.\\
\end{split}
\]
So $\CP_2=j_*L[1]\oplus \BQ_u^{b-1}\oplus \BQ_v^{a-1}$. Now $\CP_2$ is a perverse sheaf since  $\BQ_u^{b-1}\oplus \BQ_v^{a-1}$ and $j_*L[1]$ are by Lemma \ref{perv3} and Proposition \ref{perv4}. 

Applying the octahedral axiom to \eqref{beta}, we have
\begin{equation} \label{P31}
\begin{tikzcd}
   & & & H[-2]\\
   Rh_*\BQ_X\arrow[rd]\arrow[rr,"\beta"]& & \BQ_{\BR^3}[-3]\arrow[r]\arrow[ru] & Cone(\beta)\arrow[u,dotted] \\
    & (\BQ_{\BR^3}\oplus H)[-3]\arrow[ru]\arrow[rrd]& & \\
          & & & \tau_{\le2}Rh_*\BQ_X[1].\arrow[uu,dotted]
\end{tikzcd}
\end{equation}
We define $\iota_1$ as the composition
\begin{equation}\label{iota1}
Cone(\alpha)[-1]\to \tau_{\le2}Rh_*\BQ_X \to Cone(\beta)[-1],
\end{equation}
where $Cone(\alpha)[-1]\to \tau_{\le2}Rh_*\BQ_X$ is the natural morphism in the distinguished triangle defining $Cone(\alpha)$ and $\tau_{\le2}Rh_*\BQ_X \to Cone(\beta)[-1]$ is the dotted arrow obtained form \eqref{P31}. Apply the octahedral axiom to \eqref{iota1}, we have
\[
\begin{tikzcd}
   & & & H[-3]\\
   Cone(\alpha)[-1]\arrow[rd]\arrow[rr,"\iota_1"]& & Cone(\beta)[-1]\arrow[r]\arrow[ru] & \CP_3[-3]\arrow[u,dotted] \\
    & \tau_{\le2}Rh_*\BQ_X\arrow[ru]\arrow[rrd]& & \\
          & & & j_*\Lambda^2 L[-2].\arrow[uu,dotted]
\end{tikzcd}
\]
The extension class of the dotted distinguished triangle
\[
j_*\Lambda^2L[1]\to \CP_3\to H\to
\]
is 
\[
\Ext^1(H,j_*\Lambda^2L)=\Hom(H,j_*\Lambda^2 L[2] )=\Hom(j^*H,\Lambda^2L[2])=\Hom(0,\Lambda^2L)=0.
\]
So $\CP_3=j_*\Lambda^2L[1]\oplus H$ is perverse since both $H$ and $j_*\Lambda^2L[1]$ are perverse sheaves by Proposition \ref{perv4}.  Finally, the morphism $\iota_2:Cone(\beta)[-1]\to Rh_*\BQ_X$ is defined to be the connection morphism in the distinguished triangle
\[
Rh_*\BQ_X\to \BQ_{\BR^3}[-3]\to Cone(\beta)\to
\]
defining $Cone(\beta)$, and hence $\CP_4=\BQ_{\BR^3}[1]$, which is a perverse sheaf by Lemma \ref{perv3}. 
\end{proof}

Comparing with Theorem \ref{leray split}, the difference between the perverse truncation and the standard truncation of $Rh_*\BQ_X$ the indices where the subquotients $\BQ_u^{b-1},\BQ_v^{a-1},H$ live are shifted. They are supported in real codimension 2 strata and their indices in the perverse truncation are reduced by 1. Similar to Theorem \ref{leray split}, we have $E_2$ degeneracy for the perverse spectral sequence.

\begin{prop} \label{perv split}
   The perverse spectral sequence 
   \[
   ^\Fp E_{2}^{p,q}=\BH^{p}(\BR^3,{^\Fp\CH}^q(Rh_*\BQ_X))\Longrightarrow H^{p+q}(X,\BQ)
   \]
   degenerates at $E_2$-page. 
\end{prop}

\begin{proof}
   By Proposition \ref{sm}, Corollary \ref{s}, Theorem \ref{pervtr} and Proposition \ref{coh}, we have
   \[
   \sum_{p+q=k}\dim\BH^{p}(\BR^3,{^\Fp\CH}^q(Rh_*\BQ_X))= \dim H^{k}(X,\BQ),\,\forall\, k\ge0.
   \]
   So all the differentials $d^{p,q}_r$ are 0 for $r\ge 2$ and $p,q\ge0$. We conclude that the perverse spectral sequence degenerates at $E_2$-page.
\end{proof}

Combining the results above, we have the following description of the perverse filtrations.

\begin{thm} \label{PH}
Let \[
X_{a,b}=\{(x,x',y,y',z)\mid xx'=z^a+1,yy'=z^b+1,\,z\neq0\}
\]
be the cluster variety and
\begin{equation*}
\begin{split}
h:X_{a,b}&\to \BR^3\\
(x,x',y,y',z)&\mapsto(|x|^2-|x'|^2,|y|^2-|y'|^2,\log |z|).
\end{split}
\end{equation*} 
be the real analytic proper map. If $X_{a,b}$ is smooth, i.e. $ord_2(a)\neq ord_2(b)$, the non-zero graded pieces of the perverse filtration associated with $h:X\to\BR^3$ is listed below.
\begin{center}
\begin{tabular}{c|cccc}
$X_{a,b}$ & $P_0$ & $\textup{Gr}_1^P$ & $\textup{Gr}_2^P$ & $\textup{Gr}_3^P$ \\    
\hline
$H^0$ & $1$ &   &   &  \\
$H^1$ &   & $1$ &   &   \\
$H^2$ &   & $a+b-2$& $2$ &  \\
$H^3$ &   &   & $a+b-1$ & $1$
\end{tabular}
\end{center}
If $X_{a,b}$ is singular, i.e. $ord_2(a)= ord_2(b)$, the non-zero graded pieces of the perverse filtration associated with $h:X\to\BR^3$ is listed below.
\begin{center}
\begin{tabular}{c|cccc}
$X_{a,b}$ & $P_0$ & $\textup{Gr}_1^P$ & $\textup{Gr}_2^P$ & $\textup{Gr}_3^P$ \\    
\hline
$H^0$ & $1$ &   &   &  \\
$H^1$ &   & $1$ &   &   \\
$H^2$ &   & $a+b-2$& $2$ &  \\
$H^3$ &   &   & $a+b-(a,b)-1$ & $1$
\end{tabular}
\end{center}
\end{thm}

\begin{proof}
By Proposition \ref{byhand} and Proposition \ref{perv split}, we have 
\[
\textup{Gr}_q^PH^p(X,\BQ)\cong \BH^p(X,{^\Fp}\CH^k(Rh_*\BQ_X)).
\]
The dimensions of graded pieces are calculated by taking the hypercohomology of the direct summands
\[
^\Fp\CH^k(Rh_*\BQ_X)=\begin{cases}
\BQ_{\BR^3}[1] & k=1,\\
j_*L[1]\oplus\BQ_u^{b-1}\oplus\BQ_v^{a-1} &k=2,\\
j_*\Lambda^2L[1]\oplus H& k=3,\\
\BQ_{\BR^3}[1] & k=4, \\
0 & \textup{otherwise}.
\end{cases}
\]
obtained in Theorem \ref{pervtr} and use Proposition \ref{coh}. 
\end{proof}

\begin{thm} \label{PIH}
Let $X_{a,b}$ and $h:X\to\BR^3$ be as in Theorem \ref{PH}. Suppose $X_{a,b}$ is singular, i.e. $ord_2(a)=ord_2(b)$. Then the non-zero graded pieces of the perverse filtration $IH^k(X)$ associated with $h$ are listed below.
\begin{center}
\begin{tabular}{c|cccc}
$X_{a,b}$ & $P_0$ & $\textup{Gr}_1^P$ & $\textup{Gr}_2^P$ & $\textup{Gr}_3^P$ \\    
\hline
$IH^0$ & $1$ &   &   &  \\
$IH^1$ &   & $1$ &   &   \\
$IH^2$ &   & $a+b+(a,b)-2$& $2$ &  \\
$IH^3$ &   &   & $a+b-(a,b)-1$ & $1$
\end{tabular}
\end{center}
\end{thm} 

\begin{proof}
   By Proposition \ref{sing}, all the $(a,b)$ singularities of $X$ are of type $A_1$. Apply the argument of \cite[Example 4.2.2]{dC3fold}, we obtain a distinguished triangle
   \[
   \BQ_X[3]\to IC_X\to\BQ_{Z_{(a,b)}}[1]\to,
   \]
   where $Z_{(a,b)}$ is the set of $(a,b)$-th roots of $-1$. Since $Z_{(a,b)}\subset h^{-1}(O)$, applying $Rh_*$ yields
   \[
   Rh_*\BQ_X[3]\to Rh_*IC_X\to\BQ_{O}^{(a,b)}[1]\to.
   \]
   Applying the perverse cohomology functor $^\Fp\CH^k$, we obtain isomorphisms 
   \[
   ^\Fp\CH^{i}(Rh_*IC_X)\cong {^\Fp}\CH^{i}(Rh_*\BQ_X[3]),\quad i\le-2\textup{ or }i\ge1
   \]
   and a long exact sequence of shifted perverse sheaves
   \[
   0\to j_*L[1]\oplus\BQ_u^{b-1}\oplus\BQ_v^{a-1}\to {^\Fp}\CH^{-1}(Rh_*IC_X)\to \BQ_O^{(a,b)}\xrightarrow{\alpha} j_*\Lambda^2 L[1]\oplus H \to {^\Fp}\CH^{0}(Rh_*IC_X)\to 0.
   \]
   We claim that the connecting map $\alpha=0$. In fact, 
   \[
   \Hom_{D(\BR^3)} (\BQ_O, j_*\Lambda^2)=\Hom_{D(U)}(j^*\BQ_O,j_*\Lambda^2)= \Hom_{D(U)}(0,j_*\Lambda^2)=0,
   \]
   and
   \[
   \Hom_{D(\BR^3)}(\BQ_O,H)=\Hom_{D(\BR^3)}(i_{O!}\BQ_O,H)=\Hom_{D(O)}(\BQ_O,i_O^!H)=H^0(i_O^!H)=0,
   \]
   where the last equality follows from applying the hypercohomology functors to the distinguished triangle
   \[
   i_{O*}i_O^!H\to H\to  Rj_*j^*H\to. 
   \]
   Consequently, we conclude that 
   \[
   {^\Fp}\CH^{0}(Rh_*IC_X)\cong j_*\Lambda^2 L[1]\oplus H={^\Fp}\CH^{0}(Rh_*\BQ_X[3])
   \]
   and 
   \[
   {^\Fp}\CH^{-1}(Rh_*IC_X)=j_*L[1]\oplus \BQ_u^{b-1}\oplus \BQ_u^{a-1}\oplus \BQ_O^{(a,b)}={^\Fp}\CH^{-1}(Rh_*\BQ_X[3])\oplus\BQ_O^{(a,b)}.
   \]
   The perverse filtration on $IH^*(X,\BQ)$ follows from Theorem \ref{PH}.
\end{proof}

\begin{rmk}
   We have several remarks on the perverse filtration associated with the map $h$.
   \begin{enumerate}
   \item The dimension of the graded pieces of the perverse filtrations on the compactly supported cohomology $H^*_c(X,\BQ)$ and on the cohomology of the minimal resolution $H^(\widetilde{X},\BQ)$ can also be computed from Theorem \ref{pervtr}. They do not match the graded pieces on the weight filtration as the $P=W$ phenomenon suggests.
   \item If one adopts the lower middle perversity instead of the upper middle perversity, then: 
   \begin{enumerate}
      \item if $X_{a,b}$ is singular, the indices in the perverse truncations are shifted by 1 globally, and hence the normalized perverse filtration keeps the same.
      \item if $X_{a,b}$ is smooth, in addition to the global shift of indices, the direct summand $\BQ_O$ in $H$ is shifted down by an additional 1. Consequently, the perverse filtration differs from the upper middle perversity case. 
   \end{enumerate}
   \end{enumerate}
\end{rmk}

\section{P=W}
In this section we prove the $P=W$ and $PI=WI$ identities for 3-dimensional isolated cluster varieties  
\[
X_{a,b}=\{(x,x',y,y',z)\mid xx'=z^a+1,yy'=z^b+1,\,z\neq0\}.
\]
As usual, the real analytic proper map to define the perverse filtration is
\begin{equation*}
\begin{split}
h:X_{a,b}&\to \BR^3\\
(x,x',y,y',z)&\mapsto(|x|^2-|x'|^2,|y|^2-|y'|^2,\log |z|).
\end{split}
\end{equation*}

The following two propositions characterize the non-top pieces in both filtrations as the cohomology classes which do not have full support.

\begin{prop} \label{Wlemma}
   Let $X^\circ\subset X$ be the open set where $|z|\neq1$, and let $j:X^\circ\to X$ be the open embedding. Then we have 
\[
W_{2i-2}H^i(X)=\ker\{H^i(X,\BQ)\to H^i(X^\circ,\BQ)\}, \,i=2,3.
\]
\end{prop}

\begin{proof}
   Let $\pi:\widetilde{X}\to X$ be the natural projection from the resolution constructed in Proposition \ref{res}. The natural isomorphism $j^*R\pi_*\BQ_{\widetilde{X}}\cong\BQ_{X^\circ}$ induces a morphism $R\pi_*\BQ_{\widetilde{X}}\to Rj_*\BQ_{X^\circ}$, and hence a morphism on cohomology $H^*(\widetilde{X},\BQ)\to H^*(X^\circ,\BQ)$, which fits into the commutative diagram
   \[
    \begin{tikzcd}
    H^*(X,\BQ) \arrow[rr,"j^*"]\arrow[rd,"\pi^*"]& & H^*(X^\circ,\BQ).\\
    & H^*(\widetilde{X},\BQ)\arrow[ru]& 
    \end{tikzcd}
   \]
   The dual of this map is the composition $H^*_c(X^\circ,\BQ)\to H^*_c(X,\BQ)\to H^*_c(\widetilde{X},\BQ)$ obtained by the $R\Gamma_c$ of the morphism $j_!\BQ_{X^\circ}\to\BQ_X\to R\pi_*\BQ_{\widetilde{X}}$. We claim that the image of $H^k_c(X^\circ,\BQ)\to H^k_c(\widetilde{X},\BQ)$ is exactly $W_{2(k-3)}H_c^k(\widetilde{X},\BQ)$. To prove this, we push forward to $\BC^*$:
   \[
   Rf_!j_!\BQ_{X^\circ}\to Rf_!\BQ_X\to Rf_!R\pi_*\BQ_{\widetilde{X}}.
   \]
   Let $\BC^\circ=\BC^*\setminus S^1$. We have the following Cartesian diagram
   \[
   \begin{tikzcd}
       X^\circ \arrow[r,"j"]\arrow[d,"f^\circ"] & X\arrow[d,"f"]\\
       C^\circ=\BC^*\setminus S^1\arrow[r,"j_\BC"] & \BC^*
   \end{tikzcd}
   \]
   where $f:X\to\BC^*$ is given by $(x,x',y,y',z)\mapsto z$ and $f^\circ$ is its restriction to $X^\circ$. We first note that $f^\circ$ is a trivial $(\BC^*)^2$-bundle. Indeed, since $|z|\neq1$, we have $z^a+1\neq0$ and $z^b+1\neq0$. Thus the variables $x,y\in\BC^*$ trivialize the fibers $(\BC^*)^2$ globally. Therefore, we have 
   \[
   Rf^\circ_!\BQ_{X^\circ}=\BQ_{\BC^\circ}[-4]\bigoplus\BQ_{\BC^\circ}^2[-3] \bigoplus\BQ_{\BC^\circ}[-2].
   \]
   By $Rj_{\BC!}\BQ_{\BC^\circ}=\BQ_{0<|z|<1}\oplus \BQ_{|z|>1}$ and \eqref{ab}, we have the commutative diagram
   \[
   \begin{tikzcd}
   Rf_!j_!\BQ_{X^\circ}\arrow[r,"\cong"]\arrow[dd]& (\BQ_{0<|z|<1}\oplus \BQ_{|z|>1})[-4]\bigoplus(\BQ_{0<|z|<1}\oplus \BQ_{|z|>1})^2[-3] \bigoplus (\BQ_{0<|z|<1}\oplus \BQ_{|z|>1})[-2]\arrow[d]\\
    & \BQ_{\BC^*}[-4]\bigoplus\BQ_{\BC^*}^2[-3] \bigoplus \BQ_{\BC^*}[-2]\arrow[d]
   \\
   Rf_!\BQ_X\arrow[r,"\cong"]& \BQ_{\BC^*}[-4]\bigoplus (Rj_{a*}\BQ_{U_a} \oplus Rj_{b*}\BQ_{U_b})[-3]\bigoplus  Rj_{a*}\BQ_{U_a}\otimes Rj_{b*}\BQ_{U_b}[-2],
   \end{tikzcd}
   \]
   where the vertical morphism on the right respects the cohomological shifts. Applying the functor $R\Gamma_c$, we have
   \[
   \begin{tikzcd}
   H^*_c(X^\circ,\BQ)\arrow[r,"\cong"]\arrow[dd]& H^*(\BC^\circ,\BQ)[-4]\bigoplus H^*(\BC^\circ,\BQ)^2[-3] \bigoplus H^*(\BC^\circ,\BQ)[-2]\arrow[d,twoheadrightarrow,"\alpha"]\\
    & H^*(\BC^*,\BQ)[-4]\bigoplus H^*(\BC^*,\BQ)^2[-3] \bigoplus H^*(\BC^*,\BQ)[-2]\arrow[d,"\beta"]
   \\
   H^*_c(X,\BQ)\arrow[r,"\cong"]\arrow[d]& H^*_{c,b}(X,\BQ)\bigoplus H^*_{c,nb}(X,\BQ)\arrow[d,"\gamma"]\\
    H^*_c(\widetilde{X},\BQ)\arrow[r,"\cong"]& H^*_{c,b}(\widetilde{X},\BQ)\bigoplus H^*_{c,nb}(\widetilde{X},\BQ),\\
   \end{tikzcd}
   \]
   where $H^*_{c,b}=\oplus_k W_{2(k-3)}H^k_c$ is the bottom weight piece in the mixed Hodge structure on the compactly supported cohomology, and $H^*_{c,nb}$ is the non-bottom piece spanned by the classes with non-full support. A simple topological calculation shows that $\alpha$ is surjective. The proof of Proposition \ref{fiberedprod} and Theorem \ref{cpt} shows that the image of $\beta$ is exactly $H^*_{c,b}$. Theorem \ref{cpt} and Proposition \ref{res} together imply that $\gamma$ is an isomorphism on $H^*_{c,b}$. So the claim is proved. 
   
   Taking the dual statement, we see that the image of every nonzero class in $\textup{Gr}^W_{2k} H^k(\widetilde{X},\BQ)$ remains nonzero via $H^k(\widetilde{X},\BQ)\to H^k(X^{\circ},\BQ)$. Furthermore, by comparing Proposition \ref{res} and Theorem \ref{s}, we know that the pullback map $H^k(\widetilde{X},\BQ)\to H^k(X,\BQ)$ induces an isomorphism on $\textup{Gr}_{2k}^W H^k$, so we conclude that 
   \[
   W_{2k-2}H^k(X)=\ker \{H^k({X},\BQ)\to H^k(X^{\circ},\BQ)\}.
   \]
\end{proof}

\begin{prop} \label{Plemma}
   Let $X^\circ\subset X$ be the open set where $|z|\neq1$, and let $j:X^\circ\to X$ be the open embedding. Then we have 
\[
P_{i-1}H^i(X)=\ker\{H^i(X)\to H^i(X^\circ)\}, \,i=2,3.
\]
\end{prop}

\begin{proof}
   We consider the following Cartesian diagram
   \[
   \begin{tikzcd}
       X^\circ\arrow[r,"j_X"]\arrow[d,"h^\circ"]& X\arrow[d,"h"]\\
       \BR^{3\circ}=\BR^3\setminus\{w=0\}\arrow[r,"j_\BR"]& \BR^3.
   \end{tikzcd}
   \]
   By Theorem \ref{5.3}, the map $h^\circ$ is a trivial $T^3$-fibration. So 
   \[
   Rh^\circ_*\BQ_{X^\circ}=\BQ_{\BR^{3\circ}}\oplus \BQ_{\BR^{3\circ}}^3[-1]\oplus \BQ_{\BR^{3\circ}}^3[-2]\oplus \BQ_{\BR^{3\circ}}[-3].
   \]
   A direct calculation shows that
   \[
   Rj_{\BR*}\BQ_{\BR^{3\circ}}=j_{\BR*}\BQ_{\BR^3}\cong\BQ_{\BR^{3+}}\oplus \BQ_{\BR^{3-}}
   \]
   where $\BR^{3+}=\{w\ge0\}$ and $\BR^{3-}=\{w\le0\}$.
   Therefore, the natural morphism 
   \[
   Rh_*\BQ_X\to Rj_{\BR*}Rh^\circ_*\BQ_{X^\circ}=j_{\BR*}Rh^\circ_*\BQ_{X^\circ}
   \]
   induces morphisms of sheaves 
   \begin{equation} \label{pervm}
   R^qh_*\BQ_{X}\to j_{\BR*}R^qh^\circ_*\BQ_{X^\circ}.
   \end{equation}
   When $q=2$, following the notation in Proposition \ref{rh2}, the morphism \eqref{pervm} becomes
   \[
   j_*\Lambda^2L\oplus \BQ_u^{b-1}\oplus \BQ_v^{a-1}\to j_{\BR*}j^*_{\BR}(j_*\Lambda^2L\oplus \BQ_u^{b-1}\oplus \BQ_v^{a-1})\cong \BQ_{\BR^{3+}}^3\oplus \BQ_{\BR^{3-}}^3,
   \]
   since $\BQ_u$ and $\BQ_v$ are support outside $\BR^{3\circ}$ and $L$ is trivial on $\BR^{3\circ}$. Therefore, the kernel of the map oncohomology
    \[
   H^0(\BR^3,R^2h_*\BQ_{X})\to H^0(\BR^3,R^2h^\circ_*\BQ_{X^\circ})
   \]
   is exactly contributed by the summand $\BQ_u^{b-1}\oplus \BQ^{a-1}_v$, which, by Theorem \ref{pervtr}, is $P_1H^2(X,\BQ)$. Now by Theorem \ref{leray split}, the morphism $H^0(\BR^3,R^2h_*\BQ_{X})\to H^0(\BR^3,R^2h^\circ_*\BQ_{X^\circ})$ is exactly the map $H^2(X,\BQ)\to H^2(X^\circ,\BQ)$. So we conclude that $q=2$ holds.

   Similarly, when $q=3$, Proposition \ref{rh3} implies that \eqref{pervm} is 
   \[
   \BQ_{\BR^3}\oplus H\to j_{\BR*}j^*_{\BR}\BQ_{\BR^3}=\BQ_{\BR^{3+}}\oplus \BQ_{\BR^{3-}},
   \]
   where $H$ maps to $0$ since it is supported outside $\BR^{3\circ}$. The kernel of \[
   H^0(\BR^3,R^3h_*\BQ_{X})\to H^0(\BR^3,R^3h^\circ_*\BQ_{X^\circ})
   \]
   is contributed by the summand $H$, which, by Theorem \ref{pervtr}, is $P_2H^3(X,\BQ)$. By Theorem \ref{leray split}, $H^0(\BR^3,R^3h_*\BQ_{X})\to H^0(\BR^3,R^3h^\circ_*\BQ_{X^\circ})$ is exactly the map $H^3(X,\BQ)\to H^3(X^\circ,\BQ)$. So the statement holds for $q=3$.
\end{proof}

\begin{thm} \label{P=W}
   The $P=W$ identity on cohomology 
   \[
P_kH^*(X_{a,b},\BQ)=W_{2k}H^*(X_{a,b},\BQ)=W_{2k+1}H^*(X_{a,b},\BQ),~~k\ge0
\]
holds, where the perverse filtration is defined by the map $h$ with respect to the upper middle perversity function.
\end{thm}

\begin{proof}
By Theorem \ref{s} and \ref{PH}, the dimensions of the graded pieces of the weight filtration match those of the perverse filtration. To prove the corresponding subspaces coincide, it suffices to prove that 
\[
P_{k-1}H^{k}(X)=W_{2k-2}H^{k}(X), ~~k=2,3.
\]
By Proposition \ref{Wlemma} and \ref{Plemma}, they are both identified with the kernel of $H^k(X)\to H^k(X^\circ)$.
\end{proof}

\begin{thm} \label{PI=WI}
   The $PI=WI$ identity on the intersection cohomology
  \[
P_kIH^*(X_{a,b},\BQ)=W_{2k}IH^*(X_{a,b},\BQ)=W_{2k+1}IH^*(X_{a,b},\BQ),~~k\ge0
\]
holds, where the perverse filtration is defined by the map $h$ with respect to the upper middle perversity function.
\end{thm}

\begin{proof}
   When $X_{a,b}$ is smooth, the statement is the same Theorem \ref{P=W}. Now suppose $X_{a,b}$ is singular, i.e. $ord_2(a)=ord_2(b)$. By Theorem \ref{s} and Theorem \ref{PIH} we see that   
   \[
   IH^*(X,\BQ)=H^*(X,\BQ)\oplus \BQ_O^{(a,b)}[-2],
   \]
   where the direct summand $\BQ_O^{(a,b)}[-2]$ is contributed by one $H^2$ class of each exceptional divisor in the resolution $\widetilde{X}_{a,b}\to X_{a,b}$, and they belongs to $P_1H^2(X_{a,b},\BQ)$ and $W_2H^2(X_{a,b},\BQ)$. Therefore, the $PI=WI$ identity holds by Theorem \ref{P=W}.
\end{proof}
 
Since the full rank case for 3-dimensional isolated cluster varieties is proved in \cite[Theorem 1.3]{Z2}, combining with Theorem \ref{P=W} and \ref{PI=WI}, we have:

\begin{cor}
The $P=W$ and $PI=WI$ identities hold for all 3-dimensional isolated cluster varieties. 
\end{cor}

\end{document}